\documentclass[reqno,10pt]{amsart}
\usepackage{ esint }

\usepackage{overpic}
\usepackage{amssymb}
\usepackage{amsmath}
\usepackage{amsthm}
\usepackage{mathrsfs}
\usepackage{mathtools}

 \usepackage{tikz}
\usetikzlibrary{shapes,snakes,spy,calc,shapes.geometric}

\usepackage{pgf}
\usepackage{color}
\usepackage{graphicx}   
\usepackage{hyperref}

\usepackage[T1]{fontenc} 
\usepackage[utf8]{inputenc}

\usepackage{amsmath}
  \usepackage{paralist}
\usepackage{graphicx}  
\usepackage{psfrag}
\usepackage{comment}
\usepackage{esvect}

\newtheorem{theorem}{Theorem}[section]
\newtheorem{corollary}[theorem]{Corollary}

\newtheorem{lemma}[theorem]{Lemma}
\newtheorem{proposition}[theorem]{Proposition}

\theoremstyle{definition}
\newtheorem{definition}[theorem]{Definition}
\newtheorem{remark}[theorem]{Remark}

\numberwithin{equation}{section}

\newcommand{\R}{\mathbb{R}}
\newcommand{\N}{\mathbb{N}}

\newcommand{\eps}{\varepsilon}

\newcommand{\ove}{\overline}

\def\diam{{\rm diam}}

\newcommand{\B}{A''}

\newcommand{\Rd}{{\R}^d}
\newcommand{\Sd}{{\mathbb{S}}^{d-1}}
\newcommand{\dx}{\, \mathrm{d} x}
\renewcommand{\dh}{\, \mathrm{d} \mathcal{H}^{d-1}}
\newcommand{\hd}{\mathcal{H}^{d-1}}
\newcommand{\Ld}{{\mathcal{L}}^d}

\newcommand{\sm}{\setminus}
\newcommand{\Mdd}{{\mathbb{M}^{d\times d}_{\rm sym}}}
\newcommand{\Mddskew}{{\mathbb{M}^{d\times d}_{\rm skew}}}

\newcommand{\BBB}{\color{black}}

\newcommand{\EEE}{\color{black}}

\usepackage{latexsym,amsfonts}
\usepackage{mathrsfs}
\usepackage{centernot}
\usepackage{paralist}

\usepackage{eso-pic}  
\usepackage{pifont}
\usepackage{fixmath}

\numberwithin{equation}{section}

\begin{document}

\title[Integral representation in linear elasticity  with surface discontinuities]{Integral representation for  energies  in linear elasticity  with surface discontinuities}

\author{Vito Crismale}
\address[Vito Crismale]{CMAP, \'Ecole Polytechnique, 91128 Palaiseau Cedex, France}
\email[Vito Crismale]{vito.crismale@polytechnique.edu}

\author{Manuel Friedrich}
\address[Manuel Friedrich]{Applied Mathematics M\"unster, University of M\"unster\\
Einsteinstrasse 62, 48149 M\"unster, Germany.}
\email{manuel.friedrich@uni-muenster.de}

\author{Francesco Solombrino}
\address[Francesco Solombrino]{Dip. Mat. Appl. ``Renato Caccioppoli'', Univ. Napoli ``Federico II'', Via Cintia, Monte S. Angelo
80126 Napoli, Italy}
\email{francesco.solombrino@unina.it}



\begin{abstract}
In this paper we prove an integral representation formula for a general class of energies defined on the space of generalized special functions of bounded deformation ($GSBD^p$) in arbitrary space dimensions. Functionals of this type naturally arise  in the modeling of linear elastic solids with surface discontinuities including phenomena as fracture, damage,  surface tension between different elastic phases, or material voids. Our approach is based on the global method for relaxation devised in \cite{BFM} and a recent Korn-type inequality in $GSBD^p$ \cite{CCS}.  Our general strategy also allows to generalize integral representation results in $SBD^p$, obtained in dimension two  \cite{Conti-Focardi-Iurlano:15}, to higher dimensions, and to revisit results in the framework of generalized special functions of bounded variation ($GSBV^p$). 
\end{abstract}
\maketitle

\section{Introduction}

 Integral representation results  are a fundamental tool in the abstract theory of variational limits  by $\Gamma$-convergence  or in relaxation problems  (see \cite{DMMod80}). The topic has  attracted  widespread attention in the mathematical community over the last decades,  with  applications  in various contexts, such as homogenization, dimension reduction, or  atomistic-to-continuum approximations.   In this paper we contribute to this topic by proving  an integral representation result for a general class of energies arising in the modeling of linear elastic solids with surface discontinuities.

Integral  representation theorems have been provided with increasing generality, ranging from functionals defined on Sobolev spaces    \cite{Alb94,   ButDM80, ButDM85JMPA, ButDM85, DeG75, Sbo75}     to those defined on spaces of functions of bounded variation  \cite{BouDM93, CarSbo79,  DM80,    BFM},  in  particular  on the subspace $SBV$ of special functions of bounded variation \cite{BFLM, BraChP96, Braides-Defranceschi-Vitali}  and on piecewise constant functions \cite{AmbrosioBraides}.    In recent years, this analysis has been further improved to deal with
functionals  and variational limits  on   $GSBV^p$  (generalized   special  functions of bounded variation  with $p$-integrable bulk density),    which is   the natural energy space for the variational description of many problems with free discontinuities,  see  among others   \cite{BacBraZep18, BacCicRuf19, barfoc, BarLazZep16, Caterina, focgelpon07, Fri19CVPDE}. A very general method for dealing with all the abovementioned  classes of functionals, the so-called \emph{global method for relaxation}, has been    developed by {\sc Bouchitt\'{e}, Fonseca, Leoni, and Mascarenhas}       in \cite{BFLM, BFM}.  It  essentially consists in comparing asymptotic Dirichlet problems
on small balls with different boundary data depending on the local properties of the functions and allows to characterize energy densities in terms of cell formulas.

When coming to the variational description of rupture phenomena in  general linearly  elastic materials, however, the functional setting to be considered becomes weaker. Indeed,  problems need to be formulated in suitable subspaces of \emph{functions of bounded deformation} ($BD$ functions)   for  which the distributional symmetrized gradient  is a bounded Radon measure.

  In the mathematical description of  linear elasticity, the elastic properties are    determined  by the \emph{elastic strain}. For a solid in a (bounded) reference configuration $\Omega \subset \Rd$, whose \emph{displacement field} with respect to the equilibrium is $u\colon \Omega \to \Rd$, the elastic strain is given by the symmetrized gradient $e(u)=\tfrac12 (\nabla u + (\nabla u)^{\mathrm{T}})$.  In standard models, the corresponding   linear elastic energy is  a suitable quadratic form of $e(u)$, possibly depending on the material point,  see   e.g.\ \cite[Section~2.1]{FraMar98}. However,  this is often generalized to the case of  $p$-growth  for a power $p>1$  \cite[Sections~10, 11]{Hut}.  The presence of surface discontinuities is related to several dissipative phenomena, such as cracks, surface tension between different elastic phases, or internal cavities.  In  the energetic description,  this is represented  by   a term concentrated on the \emph{jump set} $J_u$. This set is characterized by the property that for $x \in J_u$, when blowing up around $x$, the jump set approximates a hyperplane with  normal $\nu_u(x)\in \Sd$ and
the displacement field is close to two suitable values $u^+(x)$, $u^-(x) \in \Rd$ on the two sides of the material with respect to  this   hyperplane.

 Prototypical examples of functionals described above are energies which are   controlled from above and below by suitable multiples of
\begin{equation}\label{control1}  \tag{1}
\int_\Omega |e(u)|^p \dx + \int_{J_u \cap \Omega} (1+|[u]|) \dh,
\end{equation}
 where   $[u](x)=u^+(x)-u^-(x)$ denotes the jump opening,  or  which are controlled by multiples of \emph{Griffith's energy} \cite{Griffith}  
\begin{equation}\label{control2}  \tag{2}
\int_\Omega |e(u)|^p \dx + \hd(J_u \cap \Omega).
\end{equation}
 Whereas in case \eqref{control1}  the energy space is  be given by $SBD^p$, a subspace of $BD$, problems with control of type \eqref{control2} are naturally formulated on   \emph{generalized special functions of bounded deformation} $GSBD^p$, introduced by {\sc Dal Maso} \cite{DM}. (We refer to Section~\ref{sec: BD} for more details.)  The only available  integral representation result in this context is due to {\sc Conti, Focardi, and Iurlano}  \cite{Conti-Focardi-Iurlano:15}  who considered  variational functionals  controlled locally in terms of \eqref{control1}  in dimension  $d=2$.   Let us mention that the behavior is quite different if   linear growth on the symmetrized gradient is assumed  (corresponding to $p=1$),  as suited for the description of  plasticity.  In that case,  representation results in the framework  of $BD$   have been obtained, for instance, in \cite{BarFonToa00, EboToa03} and \cite{CarFocVG19} (see also \cite{DePRin16, Rin11ARMA}, containing essential tools for the proof).

 The goal of the present article is twofold: we generalize the results of  \cite{Conti-Focardi-Iurlano:15} for energies with control of type  \eqref{control1}  to arbitrary space dimensions and, more importantly, we extend the theory to encompass also problems of the form \eqref{control2},  which are  most relevant from an applicative viewpoint.  Indeed, already in dimension two, the extension of \cite{Conti-Focardi-Iurlano:15}    to the case where only a control of  type \eqref{control2} is available is no straightforward task.    This is   a  fundamental  difference with respect to the $BV$-theory where problems for  generalized functions of bounded variation can be reconducted to $SBV$ by a \emph{perturbation trick} (see for instance \cite{Caterina}):  one considers a small perturbation of the functional, depending on the jump opening, to represent functionals on  $SBV^p$.  Then, by letting the perturbation parameter vanish and by truncating functions suitably,  the representation can be extended to  $GSBV^p$.  Unfortunately,  the  trick of  reducing problem \eqref{control2}  to \eqref{control1}  is not expedient  in the linearly elastic context and   does not allow to deduce an integral representation result in $GSBD^p$ from the one in $SBD^p$. This is mainly due to the fact that,  given a control only on the symmetrized gradient,  it is in principle not possible to use smooth truncations to decrease the energy up to a small error.

Let us also remark that,  while in the majority of integral representation results in $BV$ and $BD$    the $L^1$-topology was considered,   this is not the right choice when only a lower bound of the form \eqref{control2} is at hand. Indeed, in this case, the available compactness results \cite{Crismale1, DM} have been established with respect to the topology of the convergence in measure. This latter is also the topology where  recently  an integral representation result for the subspace $PR(\Omega)$ of piecewise rigid functions has been proved in \cite{FM}.

In our main result (Theorem~\ref{theorem: PR-representation}), we prove an integral representation for 
\emph{variational
functionals} $\mathcal{F}\colon GSBD^p(\Omega) \times \mathcal{B}(\Omega) \to [0,+\infty)$ 
($\mathcal{B}(\Omega)$ 
denoting the Borel subsets of $\Omega$) that  satisfy the standard abstract conditions to be Borel measures in the second argument, lower semicontinuous with respect to convergence in measure, and local in the first argument.  Moreover, we require control of type  \eqref{control2}, localized   to any $B \in \mathcal{B}(\Omega)$.

 Let us comment on the proof strategy.  We follow the general approach of the \emph{global method for relaxation} provided  in \cite{BFLM, BFM}  for variational functionals in $BV$.    The proof strategy recovers the integral bulk and surface densities as blow-up limits of cell minimization formulas. The steps to be performed are the following: 
\begin{itemize}
\item one first shows that, for fixed $u \in GSBD^p(\Omega)$, the set function $\mathcal{F}(u, \cdot)$ is asymptotically equivalent to its minimum $\mathbf{m}_{\mathcal{F}}(u, \cdot)$ over competitors attaining the same boundary conditions  as  $u$ \BBB  on the boundaries of small balls centered in $x_0 \in \Omega$ with vanishing radii. \EEE With this we mean  that the two quantities have the same  Radon-Nikodym  derivative with respect to $\mu:=\mathcal{L}^d\lfloor_{\Omega} + \mathcal{H}^{d-1}\lfloor_{J_u \cap \Omega}$ (Lemma \ref{lemma: G=m});
\item one then proves that  the Radon-Nikodym  derivative $\frac{\mathrm{d}\mathbf{m}_{\mathcal{F}}(u, \cdot)}{\mathrm{d}\mu}$ only depends on $x_0$,  the value $u(x_0)$,  and the (approximate) gradient $\nabla u(x_0)$ at a Lebesgue point $x_0$, while at a jump point $x_0$ it is uniquely determined by the one-sided traces $u^+(x_0)$, $u^-(x_0)$ and the normal vector $\nu_u(x_0)$ to $J_u$ in $x_0$ (Lemmas \ref{lemma: minsame} and \ref{lemma: minsame2}).
\end{itemize}
When dealing with all  of  the abovementioned issues, a key ingredient  is  given  by
a \emph{Korn-type inequality for special functions of bounded deformation}, established  recently  by {\sc Cagnetti, Chambolle, and Scardia}  \cite{CCS},   which  generalizes   a two-dimensional result in \cite{Conti-Focardi-Iurlano:15} (see also \cite{Friedrich:15-3}) to arbitrary dimension.  It provides a control of the full gradient in terms of the symmetrized gradient, up to an exceptional set whose perimeter has a surface measure comparable to that of the discontinuity  set.  In particular,  this estimate is used to   approximate  the function $u$ with functions $u_\eps$, which have Sobolev regularity in a ball (around a Lebesgue point), or in half-balls oriented by the jump normal (around a jump point), and which converge to the purely elastic competitor $u(x_0)+\nabla u(x_0)(\cdot-x_0)$, or the two-valued function with values $u^-(x_0)$ and $u^+(x_0)$, respectively. This is done in Lemmas \ref{lemma: blow up} and \ref{le:blowupJumpPoints}, respectively, and is used for proving Lemmas \ref{lemma: minsame} and \ref{lemma: minsame2}.   Let us mention that this application of the Korn-type inequality is similar to the one in dimension two \cite{Conti-Focardi-Iurlano:15} (with the topology of convergence in measure in place of $L^1$), and constitutes the counterpart of the $SBV$-Poincar\'e inequality \cite{DeGCarLea} used in the $SBV$-case \cite{BFLM}. \BBB We also point out that our construction for approximating two-valued functions in Lemma \ref{lemma: minsame2} slightly differs from the ones in \cite{BFLM, Conti-Focardi-Iurlano:15} in order  to fix a possible flaw contained in these proofs, see Remark \ref{rem: cons} for details. \EEE

%

 In contrast to \cite{Conti-Focardi-Iurlano:15}, the Korn inequality is also used in the proof of  Lemma \ref{lemma: G=m}:  at this point, one needs to show that functions of the form  $v^\delta:=\sum v_i^\delta \chi_{B_i^\delta}$ approximate $u$ in the topology of the convergence in measure, where  $B_i^\delta$ is a fine cover of a given set with disjoint balls of radius smaller than $\delta$ and $v_i^\delta$  denote  minimizers  for $\mathbf{m}_{\mathcal{F}}(u, B_i^\delta)$.  In \cite{Conti-Focardi-Iurlano:15}, the lower bound in \eqref{control1} allows to control the distributional symmetrized gradient ${\rm E}u$ which along with a scaling argument and the classical Korn-Poincar\'e inequality in $BD$ (see \cite[Theorem 2.2]{Tem}) shows that $v_i^\delta$ is close to $u$ on each $B_i^\delta$. (In \cite{BFLM}, the  $SBV$-Poincar\'e  inequality is used.) Our  weaker lower bound of the form \eqref{control2}, however, calls for novel arguments and we use the Korn-type inequality to show that $v_i^\delta$    are close to $u$ in $L^p$ up to  exceptional  sets $\omega_i^\delta$ whose  volumes  scale like $\delta(\mathcal{F}(u, B_i^\delta)+\mu(B_i^\delta))$.

We also point out that,  if instead  a control of the type \eqref{control1} is assumed,   the  arguments leading to Theorem \ref{theorem: PR-representation} can be successfully adapted to extend the result  for functionals on $SBD^p$  (see \cite{Conti-Focardi-Iurlano:15}) to  arbitrary space dimensions,   see Theorem \ref{theorem: SBD-representation}.   This is done  by  exploiting the stronger blow-up properties of $SBD$ functions.  We   note that,  in principle, this result could be also obtained  by  adapting the arguments in \cite{Conti-Focardi-Iurlano:15} to higher dimension   by employing   the Korn inequality  \cite{CCS}. We however preferred to give a self-contained proof of Theorem \ref{theorem: SBD-representation}, which requires only slight modifications of the arguments  used for Theorem \ref{theorem: PR-representation} and nicely illustrates the differences between $SBD^p$ and its generalized space.

For a related purpose, \BBB in Section \ref{sec: appendix} \EEE we discuss how our arguments can also provide a direct proof for integral representation results  on $GSBV^p$,  if a local control on the full deformation gradient of the form
\begin{equation*}
\int_\Omega |\nabla u|^p \dx + \hd(J_u \cap \Omega)
\end{equation*}
 is given, see Theorem \ref{theorem: PR-representation-gsbv}.   In particular, no  perturbation or  truncation arguments  are  needed  in the proof.   Therefore, we believe that this provides a new perspective and a slightly simpler approach to integral representation results in  $GSBV^p$ without necessity of  the \emph{perturbation trick} discussed before, relying on the $SBV$ result.  Let us, however, mention that in \cite{Caterina} a more general growth condition from above is considered: dealing with such a condition would instead require a truncation method in the proof.


%

\BBB We close the introduction by mentioning that in a subsequent work \cite{FPS} we use the present result to obtain integral representation of $\Gamma$-limits for sequences of   energies  in linear elasticity  with surface discontinuities. There, we additionally 
characterize the bulk and surface densities as blow-up limits of cell minimization formulas where the minimization is not performed on  $GSBD^p$ but more specifically  on Sobolev functions (bulk density) and piecewise rigid functions \cite{FM} (surface density). The latter characterization particularly allows to identify integrands of relaxed functionals and to treat homogenization problems. \EEE

 The paper is organized as follows. In Section \ref{sec: main} we present our main integral representation result in $GSBD^p$. Section \ref{sec: prel} is devoted to some preliminaries about the function space. In particular, we present the Korn-type inequality established in \cite{CCS} and prove a fundamental estimate. Section~\ref{sec: global method} contains the general strategy  and the proof of Lemma \ref{lemma: G=m}. The identifications of the bulk and surface density (Lemmas \ref{lemma: minsame} and \ref{lemma: minsame2}) are postponed to Sections \ref{sec:bulk} and  \ref{sec:surf}, respectively. In Section \ref{app: sbd}  we describe the modifications necessary to obtain the $SBD^p$-case. Finally, in  \BBB Section \ref{sec: appendix} \EEE  we explain how our method can be used to establish an integral representation result in $GSBV^p$.

\section{The integral representation result}\label{sec: main}

 In this section we present our main result. We start with some basic notation.   Let $\Omega \subset \R^d$  be  open, bounded with Lipschitz boundary. Let $\mathcal{A}(\Omega)$ be the  family of open subsets of $\Omega$,  and denote by   $\mathcal{B}(\Omega)$  the family of Borel sets contained in $\Omega$.  For every $x\in \Rd$ and $\eps>0$ we indicate by $B_\eps(x) \subset \Rd$ the open ball with center $x$ and radius $\eps$. For $x$, $y\in \Rd$, we use the notation $x\cdot y$ for the scalar product and $|x|$ for the  Euclidean  norm.   Moreover, we let   $\Sd:=\{x \in \Rd \colon |x|=1\}$  and we denote  by $\mathbb{M}^{d \times d}$   the set of $d\times d$ matrices. The $m$-dimensional Lebesgue measure of the unit ball in $\R^m$ is indicated by $\gamma_m$ for every $m \in \N$.   We denote by $\Ld$ and $\mathcal{H}^k$ the $d$-dimensional Lebesgue measure and the $k$-dimensional Hausdorff measure, respectively.  

For definition and properties of the space $GSBD^p(\Omega)$,  $1 < p < \infty$, we refer the reader to \cite{DM}. Some relevant properties are collected in Section \ref{sec: prel} below. In particular,  the approximate gradient is denoted by $\nabla u$ (it is well-defined, see Lemma \ref{lemma: approx-grad}) and the (approximate) jump set is denoted by $J_u$  with corresponding normal $\nu_u$ and one-sided limits $u^+$ and $u^-$.  We also define $e(u) = \frac{1}{2} (  \nabla u + (\nabla u)^{\mathrm{T}})$.

We consider functionals $\mathcal{F}\colon GSBD^p(\Omega) \times \mathcal{B}(\Omega) \to  [0,+\infty)$ with the following general assumptions: \begin{itemize}
\item[(H$_1$)]  $\mathcal{F}(u,\cdot)$ is a Borel measure for any $u \in  GSBD^p(\Omega)$, 
\item[(H$_2$)]  $\mathcal{F}(\cdot,A)$ is lower semicontinuous with respect to convergence in measure on $\Omega$ for any $A \in \mathcal{A}(\Omega)$,
\item[(H$_3$)]   $\mathcal{F}(\cdot, A)$ is local for any $A \in \mathcal{A}(\Omega)$, in the sense that if $u,v \in GSBD^p(\Omega)$ satisfy $u=v$ a.e.\ in $A$, then $\mathcal{F}(u,A) = \mathcal{F}(v,A)$,
\item[(H$_4$)]  there exist $0 < \alpha  < \beta $ such that for any $u \in GSBD^p(\Omega)$ and $B \in \mathcal{B}(\Omega)$  we have
$$\alpha \bigg(\int_{ B} |e(u)|^p  \dx    +   \mathcal{H}^{d-1}(J_u \cap B)\bigg) \le \mathcal{F}(u,B) \le \beta \bigg(\int_{ B} (1 + |e(u)|^p)    \dx  +   \mathcal{H}^{d-1}(J_u \cap B)\bigg).$$ 
\end{itemize}
We now formulate the main result of this article addressing integral representation of functionals $\mathcal{F}$ satisfying {\rm (${\rm H_1}$)}--{\rm (${\rm H_4}$)}. To this end, we introduce some further notation: for every $u \in GSBD^p(\Omega)$ and $A \in \mathcal{A}(\Omega)$ we define 
\begin{align}\label{eq: general minimization} 
\mathbf{m}_{\mathcal{F}}(u,A) = \inf_{v \in GSBD^p(\Omega)} \  \lbrace \mathcal{F}(v,A)\colon \ v = u \ \text{ in a neighborhood of } \partial A \rbrace.
\end{align}
For $x_0 \in \Omega$, $u_0 \in \R^d$, and $\xi \in  \mathbb{M}^{d \times d}  $ we introduce the functions  $\ell_{x_0,u_0,\xi}\colon \R^d \to \R^d$ by 
\begin{align}\label{eq: elastic competitor}
\ell_{x_0,u_0,\xi}(x) =  u_0 + \xi (x-x_0). 
\end{align}
Moreover, for $x_0 \in \Omega$, $a,b \in \R^d$, and  $\nu \in \mathbb{S}^{d-1}$ we introduce  $u_{x_0,a,b,\nu} \colon \R^d \to \R^d$ by 
\begin{align}\label{eq: jump competitor}
u_{x_0,a,b,\nu}(x) = \begin{cases}  a & \text{if } (x-x_0) \cdot \nu > 0,\\ b & \text{if }  (x-x_0) \cdot \nu < 0. \end{cases} 
\end{align}

 In this paper, we will prove the  following result.

\begin{theorem}[Integral representation in $GSBD^p$]\label{theorem: PR-representation}
Let $\Omega \subset \R^d$ be open, bounded with Lipschitz boundary and suppose that  $\mathcal{F}\colon GSBD^p(\Omega)  \times \mathcal{B}(\Omega) \to [0,+\infty)$ satisfies {\rm (${\rm H_1}$)}--{\rm (${\rm H_4}$)}. Then 
$$\mathcal{F}(u,B) = \int_B f\big(x,u(x),\nabla u(x)\big)  \, {\rm d}x +    \int_{J_u\cap  B} g\big(x,u^+(x),u^-(x),\nu_u(x)\big)\,  {\rm d}  \mathcal{H}^{d-1}(x)$$
for all $u \in  GSBD^p(\Omega)$   and    $B \in \mathcal{B}(\Omega)$, where $f$ is given  by
\begin{align}\label{eq:fdef}
f(x_0,u_0,\xi) = \limsup_{\eps \to 0} \frac{\mathbf{m}_{\mathcal{F}}(\ell_{x_0,u_0,\xi},B_\eps(x_0))}{\gamma_d\eps^{d}}
\end{align}
for all $x_0 \in \Omega$, $u_0 \in \R^d$, $\xi \in \mathbb{M}^{d \times d}$, and $g$ is given by 
\begin{align}\label{eq:gdef}
g(x_0,a,b,\nu) = \limsup_{\eps \to 0} \frac{\mathbf{m}_{\mathcal{F}}(u_{x_0,a,b,\nu},B_\eps(x_0))}{\gamma_{d-1}\eps^{d-1}}
\end{align}
for all $  x_0  \in \Omega$,  $a,b \in \R^d$, and $\nu \in \mathbb{S}^{d-1}$. 
\end{theorem}

\begin{remark}
{\normalfont
 We proceed with some remarks on the result. 

\noindent (i) In general, if $f$ is not convex in $\xi$, in spite of the growth conditions (H$_4$), the functional may fully depend on $\nabla u$ and not just on the symmetric part $e(u)$. We refer to \cite[Remark~4.14]{Conti-Focardi-Iurlano:15} for an example in this direction. 
  
\noindent (ii) As  $\mathcal{F}$ is lower semicontinuous on $W^{1,p}$ \BBB with respect to weak convergence, \EEE the integrand $f$ is quasiconvex \cite{Morrey}. Since   $\mathcal{F}$ is lower semicontinuous on piecewise rigid functions, the integrand $g$ is $BD$-elliptic \cite{FMM}  (at least if one can ensure, for instance, that $g$ has a continuous dependence in $x$). A  fortiori, $g$ is $BV$-elliptic \cite{AmbrosioBraides2}. 
  
\noindent (iii)  If the functional $\mathcal{F}$ additionally satisfies $\mathcal{F}(u+a,A) = \mathcal{F}(u,A)$ for all affine functions  $a\colon \R^d \to \R^d$ with $e(a) = 0$, then there are two functions $f\colon \Omega \times  \mathbb{M}^{d \times d}  \to [0,+\infty)$ and $g\colon \Omega \times \R^d \times \mathbb{S}^{d-1} \to [0,+\infty)$ such that 
$$\mathcal{F}(u,B) = \int_B f\big(x,e(u)(x)\big)  \, {\rm d}x +    \int_{J_u\cap  B} g\big(x,[u](x),\nu_u(x)\big)\,  {\rm d}  \mathcal{H}^{d-1}(x), $$
where $[u](x) := u^+(x)-u^-(x)$.

\noindent (iv) A variant of the proof shows that, in the minimization problems \eqref{eq:fdef}--\eqref{eq:gdef}, one may replace balls $B_\eps(x_0)$ by cubes $Q^\nu_\eps(x_0)$ with sidelength $\eps$, centered at $x_0$, and two faces orthogonal to  $\nu=\nu_u(x_0)$.  
\noindent (v) An analogous result holds on the space $GSBV^p(\Omega;\R^m)$ for $m \in \N$. We refer to \BBB Section \ref{sec: appendix} \EEE for details. 

}
\end{remark}

We will additionally discuss the minor modifications needed in order to deal with functionals $\mathcal{F}\colon SBD^p(\Omega) \times \mathcal{B}(\Omega) \to  [0,+\infty)$ satisfying (H$_1$)--(H$_3$) and
\begin{itemize}
\item[(H$_4^\prime$)]  there exist $0 < \alpha  < \beta $ such that for any $u \in SBD^p(\Omega)$ and $B \in \mathcal{B}(\Omega)$  we have
$$\alpha \Big(\int_{ B  } |e(u)|^p  \dx    +   \int_{J_u \cap B}\!\!(1+|[u]|)\,\mathrm{d}\mathcal{H}^{d-1}\Big) \le \mathcal{F}(u,B) \le \beta \Big(\int_{ B  } (1 + |e(u)|^p)    \dx  +   \int_{J_u \cap B}\!\!(1+|[u]|)\,\mathrm{d}\mathcal{H}^{d-1}\Big).$$ 
\end{itemize}
In this case, $SBD^p(\Omega)$ (see Subsection \ref{sec: BD}) is the natural energy space for $\mathcal{F}$.  Furthermore,  sequences of competitors with bounded energy, which are converging in measure, are additionally $L^1$-convergent if we assume (H$_4^\prime$), due to the classical Korn-Poincar\'e inequality in $BD$ \BBB (see \cite[Theorem II.2.2]{Tem}). \EEE Hence, in this latter  case,  (H$_2$) is equivalent to requiring lower semicontinuity with respect to the $L^1$-convergence. 
The statement of the result in this setting, as well as of the changes needed in the proofs, will be given in Section~\ref{app: sbd}.

\section{Preliminaries}\label{sec: prel}

We start  this preliminary section  by introducing some further notation. For $E \subset \R^d$, $\eps>0$, and $x_0 \in \R^d$ we set
\begin{align}\label{eq: shift-not}
E_{\eps,x_0} := x_0 + \eps (E - x_0).    
\end{align} 
 The diameter of $E$ is indicated by ${\rm diam}(E)$.  Given two sets $E_1,E_2 \subset \R^d$, we denote their symmetric difference by $E_1 \triangle E_2$.   We write $\chi_E$ for the  characteristic  function of any $E\subset  \R^d$, which is 1 on $E$ and 0 otherwise.  If $E$ is a set of finite perimeter, we denote its essential boundary by $\partial^* E$,   see  \cite[Definition 3.60]{Ambrosio-Fusco-Pallara:2000}.   
 We  denote the set of symmetric and skew-symmetric matrices by  $\Mdd$ and  $\Mddskew$, respectively.   
 
%

\subsection{$BD$ and $GBD$ functions}\label{sec: BD}
 Let $U\subset \Rd$  be open. 
A function $v\in L^1(U;\Rd)$ belongs to the space of \emph{functions of bounded deformation},  denoted by $BD(U)$,   if 
 the distribution  $\mathrm{E}v := \frac{1}{2}( \mathrm{D}v + (\mathrm{D}v)^{\mathrm{T}} )$ is a bounded $\Mdd$-valued Radon measure on $U$,  
where $\mathrm{D}v=(\mathrm{D}_1 v,\dots, \mathrm{D}_d v)$ is the distributional differential.
It is well known (see \cite{ACD, Tem}) that for $v\in BD(U)$  the jump set  $J_v$ is 
\BBB countably $\hd$~-~rectifiable (in the sense of \cite[Definition~2.57]{Ambrosio-Fusco-Pallara:2000}), \EEE 
and that
\begin{equation*}
\mathrm{E}v=\mathrm{E}^a v+ \mathrm{E}^c v + \mathrm{E}^j v,
\end{equation*}
where $\mathrm{E}^a v$ is absolutely continuous with respect to $\Ld$, $\mathrm{E}^c v$ is singular with respect to $\Ld$ and such that $|\mathrm{E}^c v|(B)=0$ if $\hd(B)<\infty$, while $\mathrm{E}^j v$ is concentrated on $J_v$. The density of $\mathrm{E}^a v$ with respect to $\Ld$ is denoted by $e(v)$.

The space $SBD(U)$ is the subspace of all functions $v\in BD(U)$ such that $\mathrm{E}^c v=0$. For $p\in (1,\infty)$, we define
$SBD^p(U):=\{v\in SBD(U)\colon e(v)\in L^p(U;\Mdd),\, \hd(J_v)<\infty\}$.
For a complete treatment of $BD$ and $SBD$ functions, we refer to 
to  \cite{ACD, BCD,  Tem}.

The spaces $GBD(U)$ of \emph{generalized functions of bounded deformation} and $GSBD(U)\subset GBD(U)$  of \emph{generalized special functions of bounded deformation} have been introduced in \cite{DM} (cf.\ \cite[Definitions~4.1 and 4.2]{DM}), \BBB and are defined as follows.

\begin{definition}
Let $U\subset \Rd$ be a  bounded open set,  and let  $v\colon U\to \Rd$ be measurable. We introduce the notation  
$$
\Pi^\xi:=\{y\in \Rd\colon y\cdot \xi=0\},\qquad B^\xi_y:=\{t\in \R\colon y+t\xi \in B\} \ \ \ \text{ for any $y\in \Rd$ and $B\subset \Rd$}\,,
$$
for fixed  $\xi \in \Sd$, and for every  $t\in B^\xi_y$ we let
\begin{equation}\label{eq: vxiy}
v^\xi_y(t):=v(y+t\xi),\qquad \widehat{v}^\xi_y(t):=v^\xi_y(t)\cdot \xi\,.
\end{equation}
Then, $v\in GBD(U)$ if there exists $\lambda_v\in \mathcal{M}^+_b(U)$ such that $\widehat{v}^\xi_y \in BV_{\mathrm{loc}}(U^\xi_y)$ for $\hd$-a.e.\ $y\in \Pi^\xi$, and for every Borel set $B\subset U$ 
\begin{equation*}
\int_{\Pi^\xi} \Big(\big|\mathrm{D} {\widehat{v}}_y^\xi\big|\big(B^\xi_y\setminus J^1_{{\widehat{v}}^\xi_y}\big)+ \mathcal{H}^0\big(B^\xi_y\cap J^1_{{\widehat{v}}^\xi_y}\big)\Big)\dh(y)\leq \lambda_v(B)\,,
\end{equation*}
where
$J^1_{{\widehat{v}}^\xi_y}:=\left\{t\in J_{{\widehat{v}}^\xi_y} : |[{\widehat{v}}_y^\xi]|(t) \geq 1\right\}$.
Moreover, the function $v$ belongs to $GSBD(U)$ if $v\in GBD(U)$ and $\widehat{v}^\xi_y \in SBV_{\mathrm{loc}}(U^\xi_y)$ for 
every
$\xi \in \Sd$ and for $\hd$-a.e.\ $y\in \Pi^\xi$.
\end{definition}
\EEE

 We recall that
every $v\in GBD(U)$ has an \emph{approximate symmetric gradient} $e(v)\in L^1(U;\Mdd)$ and an \emph{approximate jump set} $J_v$ which is still 
\BBB countably $\hd$-rectifiable \EEE
(cf.~\cite[Theorem~9.1,  Theorem~6.2]{DM}).

The notation for $e(v)$ and $J_v$, which is the same as that one in the $SBD$ case, is consistent: in fact, if $v$  lies in  $SBD(U)$, the objects coincide (up to  negligible sets of points with respect to $\Ld$ and $\hd$, respectively). For 
$x\in J_v$  there exist $v^+(x)$, $v^-(x)\in \Rd$  and $\nu_v(x)\in\mathbb{S}^{d-1}$  
 such that
\begin{equation}\label{0106172148}
\lim_{\eps \to 0}\eps^{-d}\Ld\big(\{y \in B_\eps(x)\colon \pm(y-x)\cdot \nu_v(x)>0\} \cap \{|v-v^\pm(x)|>\varrho\}\big)=0
\end{equation}
 for every $\varrho>0$, and the 
function  $[v]:=v^+-v^- \colon J_v \to \Rd$ is measurable.  For $1 < p < \infty$, the   space $GSBD^p(U)$ is  given by  
\begin{equation*}
GSBD^p(U):=\{v\in GSBD(U)\colon e(v)\in L^p(U;\Mdd),\, \hd(J_v)<\infty\}.
\end{equation*}
 Any  function $v\in GSBD(U)$ with $[v]$ integrable belongs to $SBD(U)$, 
as follows from \cite[Theorem~2.9]{Crismale3} for $\mathbb{A}v={\rm E} v$ (see \cite[Remark~2.5]{Crismale3}). This corresponds to  the following proposition.
\begin{proposition}\label{prop: onlysbd}
If $v \in GSBD^p(U)$ is such that $[v] \in L^1(J_v; \Rd)$, then $v \in SBD^p(U)$.
\end{proposition}

If $U$ has Lipschitz boundary, for each $v\in GBD(U)$ the traces on $\partial U$ are well defined  (see~\cite[Theorem~5.5]{DM}),   in the sense that for $\mathcal{H}^{d-1}$-a.e.\ $x \in  \partial U  $ there exists ${\rm tr}(v)(x) \in \R^d$ such that 
\begin{align}\label{eq: trace}
\lim_{\eps \to 0}\eps^{-d} \Ld\big(U \cap B_\eps(x)\cap \{|v-{\rm tr}(v)(x)|>\varrho\}\big)=0 \quad \quad \text{ for all $\varrho >0$}.
\end{align}

\subsection{Korn's inequality and fundamental estimate}

In this subsection we discuss two important tools which will be instrumental for the proof of Theorem \ref{theorem: PR-representation}. We start by  the following Korn and Korn-Poincar\'e  inequalities  in $GSBD$ for functions with small jump sets, see  \cite[Theorem 1.1, Theorem 1.2]{CCS}.  In the following, we say that $a\colon \R^d \to \R^d$ is an \emph{infinitesimal rigid motion} if $a$ is affine with $ e(a) = \frac{1}{2}  ( \nabla a + (\nabla a)^{\mathrm{T}})  = 0$.

\begin{theorem}[Korn inequality for functions with small jump set]\label{th: kornSBDsmall}
Let $\Omega \subset \R^d$ be a bounded Lipschitz domain and let $1 < p < +\infty$. Then there exists a constant $c = c(\Omega,p)>0$ such that for all  $u \in GSBD^p(\Omega)$ there is a set of finite perimeter $\omega \subset \Omega$ with 
\begin{align}\label{eq: R2main}
\mathcal{H}^{d-1}(\partial^* \omega) \le c\mathcal{H}^{d-1}(J_u), \ \ \ \ \mathcal{L}^d(\omega) \le c(\mathcal{H}^{d-1}(J_u))^{d/(d-1)}
\end{align}
and an infinitesimal  rigid motion $a$ such that
\begin{align}\label{eq: main estmain}
\Vert u - a \Vert_{L^{p}(\Omega \setminus \omega)} + \Vert \nabla u - \nabla a \Vert_{L^{p}(\Omega \setminus \omega)}\le c \Vert e(u) \Vert_{L^p(\Omega)}.
\end{align}
Moreover, there exists $v \in W^{1,p}(\Omega;\R^d)$ such that $v= u$ on $\Omega \setminus \omega$ and
\begin{align*}
\Vert e(v) \Vert_{L^p(\Omega)} \le c \Vert e(u) \Vert_{L^p(\Omega)}.
\end{align*}
\end{theorem} 
 
Note that the result is indeed only relevant for functions with sufficiently small jump set, as otherwise one can choose $\omega = \Omega$,  and \eqref{eq: main estmain} trivially holds. Note that, in \cite{CCS}, $\mathcal{L}^d(\omega) \le c(\mathcal{H}^{d-1}(J_u))^{d/(d-1)}$ has not been shown, but it readily follows from $\mathcal{H}^{d-1}(\partial^* \omega) \le c\mathcal{H}^{d-1}(J_u)$ by the isoperimetric inequality.

 \begin{remark}[Almost Sobolev regularity, constants, and scaling invariance]\label{rem: Korn-scaling}
 (i) More precisely, in \cite{CCS} it is proved that there exists $v \in W^{1,p}(\Omega;\R^d)$ such that $v= u$ on $\Omega \setminus \omega$ and $\Vert e(v) \Vert_{L^p(\Omega)} \le c \Vert e(u) \Vert_{L^p(\Omega)}$, whence by Korn's  and Poincar\'e's  inequality in $W^{1,p}(\Omega;\R^d)$ we get 
\begin{align*}
\Vert v - a \Vert_{L^{p}(\Omega)} + \Vert \nabla v - \nabla a \Vert_{L^{p}(\Omega)}\le c \Vert e(u) \Vert_{L^p(\Omega)}
\end{align*}
for an infinitesimal rigid motion  $a$.  This directly implies \eqref{eq: main estmain}, see   \cite[Theorem 4.1, Theorem 4.4]{CCS}.

(ii) Given a collection of bounded Lipschitz domains $(\Omega_k)_k$ which   are related through bi-Lipschitzian homeomorphisms with Lipschitz constants of both the homeomorphism itself and its inverse bounded uniformly in $k$, in Theorem \ref{th: kornSBDsmall}  we   can choose a constant $c$ uniformly for all $\Omega_k$, see \cite[Remark 4.2]{CCS}.

(iii)  Recall \eqref{eq: shift-not}.  Consider a bounded Lipschitz domain $\Omega$, $\eps >0$, and $x_0 \in \R^d$. Then for each $u \in GSBD^p(\Omega_{\eps,x_0})$ we find $\omega \subset \Omega_{\eps,x_0}$ and a rigid   motion $a$ such that
$$\mathcal{H}^{d-1}(\partial^* \omega) \le  C\mathcal{H}^{d-1}(J_u), \ \ \ \ \mathcal{L}^d(\omega)\le C\big(\mathcal{H}^{d-1}(J_u)\big)^{d/(d-1)}  $$
and 
$$\eps^{-1}\Vert u - a \Vert_{L^{p}(\Omega_{\eps,x_0} \setminus \omega)} + \Vert \nabla u - \nabla a \Vert_{L^{p}(\Omega_{\eps,x_0} \setminus \omega)} \le C  \Vert e(u) \Vert_{L^p(\Omega_{\eps,x_0})},  $$
where $C= C(\Omega,p) >0$ is independent of $\eps$. This follows by a standard rescaling argument. 
\end{remark}

>From Theorem~\ref{th: kornSBDsmall}, one can also deduce that for $u\in GSBD^p(\Omega)$ the \emph{approximate gradient} $\nabla u$ exists $\mathcal{L}^d$-a.e.\ in $\Omega$,  see \cite[Corollary 5.2]{CCS}.

\begin{lemma}[Approximate gradient]\label{lemma: approx-grad}
Let $\Omega \subset \R^d$ be  open, bounded with  Lipschitz boundary,   let $1 < p < +\infty$, and $u \in GSBD^p(\Omega)$. Then for $\mathcal{L}^d$-a.e.\   $x_0 \in \Omega$  there exists a matrix  in $\mathbb{M}^{d \times d}$,   denoted by $\nabla u(x_0)$, such that
$$\lim_{\eps \to 0} \  \eps^{-d} \mathcal{L}^d\Big(\Big\{x \in B_\eps(x_0) \colon \,  \frac{|u(x) - u(x_0) - \nabla u(x_0)(x-x_0)|}{|x - x_0|}  > \varrho   \Big\} \Big)  = 0 \text{ for all $\varrho >0$}. $$    
\end{lemma}

We point out that the result in Lemma \ref{lemma: approx-grad} has already been obtained in \cite{Friedrich:15-4}  for $p=2$, as a consequence of the embedding $GSBD^2(\Omega) \subset (GBV(\Omega))^d$, see \cite[Theorem 2.9]{Friedrich:15-4}.

   To control the affine mappings appearing in Theorem \ref{th: kornSBDsmall}, we will make use of the following  elementary lemma on  affine mappings, see, e.g.,    \cite[Lemma~3.4]{FM} or \cite[Lemmas~4.3]{Conti-Iurlano:15} for similar statements.  (It is obtained by the equivalence of norms in finite dimensions and by standard rescaling arguments.) 
		
\begin{lemma}\label{lemma: rigid motion}
Let $1 \le p < + \infty$, let $x_0 \in \R^d$, and let $R, \theta >0$.   Let $a\colon \R^d \to \R^d$ be affine, defined by $a(x) = A\,x + b$ for $x \in \R^d$,   and let $E \subset B_R(x_0) \subset \R^d$ with  $\mathcal{L}^d(E) \ge \theta \mathcal{L}^d(B_R(x_0))$.  Then, there exists a constant $c_0>0$ only depending on $p$ and $\theta$  such that 
\begin{align*}
\Vert a\Vert_{L^p(B_R(x_0))} \le \gamma_d^{\frac{1}{p}} R^{\frac{d}{p}} \Vert a\Vert_{L^\infty(B_R(x_0))}  \le c_0 \Vert  a \Vert_{L^p(E)}, \quad \quad \quad    |A| \le   c_0   R^{-1 -\frac{d}{p}}  \Vert  a \Vert_{L^p(E)}.
\end{align*}
\end{lemma}

 We now proceed with   another consequence of Theorem~\ref{th: kornSBDsmall}.

\begin{corollary}\label{cor: kornSBDsmall}
Let $\Omega \subset \R^d$ be a bounded Lipschitz domain and let $1 < p < +\infty$. Then there exists a constant $ C  = C(\Omega,p)>0$ such that for all  $u \in GSBD^p(\Omega)$ with trace  $ {\rm tr}(u)   = 0$ on $\partial \Omega$  (see \eqref{eq: trace})   there is a set of finite perimeter $\omega \subset \R^d$ with 
\begin{align}\label{eq: R2main-cor}
\mathcal{H}^{d-1}(\partial^* \omega) \le  C  \mathcal{H}^{d-1}(J_u), \ \ \ \ \mathcal{L}^d(\omega) \le  C  (\mathcal{H}^{d-1}(J_u))^{d/(d-1)}
\end{align}
such that
\begin{align}\label{eq: main estmain-cor}
\Vert u \Vert_{L^{p}(\Omega \setminus \omega)} + \Vert \nabla u \Vert_{L^{p}(\Omega \setminus \omega)}\le  C  \Vert e(u) \Vert_{L^p(\Omega)}.
\end{align}
\end{corollary} 

\begin{proof}
We start by choosing a  bounded  Lipschitz domain $\Omega' \subset \R^d$ with $\Omega \subset \subset \Omega'$. Each $u \in GSBD^p(\Omega)$ with $ {\rm tr}(u)  =0$ on $\partial \Omega$ can be extended to a function $\tilde{u} \in GSBD^p(\Omega')$ by $\tilde{u} = 0 $ on $\Omega' \setminus \Omega$ such that $J_{\tilde{u}} = J_u$. We first note that it is not restrictive to assume that 
\begin{align}\label{eq: small jump}
\mathcal{H}^{d-1}(J_u) \le \Big(\frac{\mathcal{L}^d(\Omega' \setminus \Omega)}{2c} \Big)^{(d-1)/d},
\end{align}
where $c=c(\Omega',p)>0$ is the constant of Theorem \ref{th: kornSBDsmall}.  In fact, otherwise we could take $\omega = \Omega$ and the statement would be trivially satisfied since \eqref{eq: main estmain-cor} is clearly trivial and for   \eqref{eq: R2main-cor} we use that
$$\mathcal{H}^{d-1}(\partial^*\Omega) \le   C\Big(\frac{\mathcal{L}^d(\Omega' \setminus \Omega)}{2c} \Big)^{(d-1)/d}, \quad \quad \quad \quad     \mathcal{L}^d(\Omega) \le C\frac{\mathcal{L}^d(\Omega' \setminus \Omega)}{2c}     $$  
for a sufficiently large constant $C>0$ depending only  on  $\Omega$ and $\Omega'$. 

Now, consider a function $u$ satisfying \eqref{eq: small jump}. We apply Theorem \ref{th: kornSBDsmall} on $\tilde{u} \in GSBD^p(\Omega')$ and obtain a set  ${\omega} \subset \Omega' \subset \R^d$ satisfying   \eqref{eq: R2main} as well as  an infinitesimal rigid motion $a$   such that 
\begin{align}\label{eq: main estmain-NNN}
\Vert \tilde{u} - a \Vert_{L^{p}(\Omega' \setminus \omega)} + \Vert \nabla \tilde{u} -  \nabla a  \Vert_{L^{p}(\Omega' \setminus \omega)}\le c \Vert e(\tilde{u}) \Vert_{L^p(\Omega')} = c\Vert e({u}) \Vert_{L^p(\Omega)}.
\end{align}
In particular,   $\tilde{u} = 0$ on $\Omega'\setminus \Omega$ implies
\begin{align}\label{eq: rigi}
 \Vert a \Vert_{L^p( \Omega' \setminus (\Omega \cup \omega)   )} \le   c \Vert e(u) \Vert_{L^p(\Omega)}. 
 \end{align}
 By  \eqref{eq: R2main} and \eqref{eq: small jump} we get $\mathcal{L}^d(\omega) \le \frac{1}{2} \mathcal{L}^d(\Omega' \setminus \Omega)$. 
In view of \eqref{eq: rigi}, we apply Lemma \ref{lemma: rigid motion}    on $E = \Omega' \setminus (\Omega \cup \omega)$ with   $R = {\rm diam}(\Omega')$ and $\theta = \frac{1}{2} \mathcal{L}^d(\Omega' \setminus \Omega) / \gamma_d R^d$   
 to get 
$$\Vert a \Vert_{L^p(\Omega')} \le c \Vert a \Vert_{L^p(\Omega' \setminus (\Omega \cup \omega))}   \le c\Vert e(u) \Vert_{L^p(\Omega)},$$ 
and, in a  similar fashion, $  |\nabla a|  \le c\Vert e(u) \Vert_{L^p(\Omega)}$, where $c>0$ depends on $\Omega$, $\Omega'$,  and $p$.  Then, \eqref{eq: main estmain-cor} follows from \eqref{eq: main estmain-NNN}, the triangle inequality,   and the fact that $u = \tilde{u}$ on $\Omega$.  
\end{proof}

 We conclude this  subsection with another  important tool in the proof of the integral representation,  namely  a fundamental estimate in $GSBD^p$.

 \begin{lemma}[Fundamental estimate in $GSBD^p$]\label{lemma: fundamental estimate}
Let $\Omega \subset \R^d$ be open, bounded with Lipschitz boundary, and let $1 < p <+\infty$.  Let $\eta >0$ and  let $A, A', \B \in \mathcal{A}(\Omega)$ with $A' \subset \subset  A$. For every functional $\mathcal{F}$  satisfying {\rm (H$_1$)}, {\rm (H$_3$)}, and {\rm (H$_4$)}   and for every $u \in GSBD^p(A)$, $v \in GSBD^p(\B)$ there exists a function $\varphi \in C^\infty(\R^d;[0,1])$  such that  $w :=  \varphi u + (1- \varphi)v \in GSBD^p(A' \cup \B)$ satisfies
\begin{align}\label{eq: assertionfund}
{\rm (i)}& \ \ \mathcal{F} ( w, A' \cup \B) \le  (1+ \eta)\big(\mathcal{F}(u,A)  + \mathcal{F}(v, \B) \big) + M \Vert u-v \Vert^p_{L^p((A \setminus A') \cap \B)} +\eta\mathcal{L}^d(A' \cup \B), \notag \\ 
{\rm (ii)} & \ \  w = u \text{ on } A' \text{ and } w = v \text{ on } \B \setminus A,
\end{align}
where $M=M(A,A',\B,p,\eta)>0$ depends only on $A,A',\B,p,\eta$, but is independent of $u$ and $v$. Moreover, if for $\eps>0$  and $x_0 \in \R^d$ we have  $A_{\eps,x_0}, A'_{\eps,x_0}, \B_{\eps,x_0} \subset \Omega$,  then 
\begin{align}\label{eq: constant-scal}
 M(A_{\eps,x_0}, A'_{\eps,x_0},\B_{\eps,x_0},p,\eta)  = \eps^{-p}M(A,A',\B,p,\eta),
\end{align} 
 where we used the notation introduced in \eqref{eq: shift-not}. 

 The same statement holds if $\mathcal{F}$ satisfies {\rm (H$_4^\prime$)},  $u \in SBD^p(A)$, and $v \in SBD^p(\B)$.
\end{lemma}

 In the statement  above,  we  intend  that $\Vert u-v \Vert^p_{L^p((A \setminus A') \cap \B)}=+\infty$ if $u-v \notin L^p((A \setminus A') \cap \B)$.

 \begin{proof}
  The proof follows the lines of \cite[Proposition 3.1]{Braides-Defranceschi-Vitali}.   Choose $k \in \N$ such that
\begin{align}\label{eq: kdef}
k \ge \max\Big\{\frac{ 3^{p-1}\beta  }{\eta \alpha}, \frac{\beta}{\eta}   \Big\}.
\end{align}
Let $A_1,\ldots,A_{k+1}$ be open subsets of $\R^d$ with $A' \subset \subset A_1 \subset \subset \ldots \subset \subset A_{k+1} \subset \subset A$. For $i=1,\ldots,k$ let $\varphi_i\in C_0^\infty(A_{i+1};[0,1])$ with $\varphi_i=1$ in a neighborhood $V_i$ of $\overline{A_i}$. 

 Consider $u \in GSBD^p(A)$ and  $v \in GSBD^p(\B)$. We   can clearly assume that $u-v \in L^p((A \setminus A') \cap \B)$ as otherwise the result is trivial.    We   define the function $w_i = \varphi_i u + (1-\varphi_i)v \in GSBD^p(A' \cup \B)$, where $u$ and $v$ are extended arbitrarily outside $A$ and $\B$, respectively.  Letting $T_i = \B \cap (A_{i+1} \setminus \overline{A_i})$ we get by (H$_1$)  and (H$_3$) 
\begin{align}\label{eq: wi}
\mathcal{F}(w_i,A' \cup \B) &\le \mathcal{F}(u, (A' \cup \B) \cap V_i) + \mathcal{F}(v, \B \setminus {\rm supp}(\varphi_i))
+  \mathcal{F}(w_i,T_i) \notag \\
&\le \mathcal{F}(u,A) + \mathcal{F}(v,\B) + \mathcal{F}(w_i,T_i). 
\end{align} 
 For the last term, we compute  using  (H$_4$)   ($\odot$ denotes the symmetrized vector product) 
\begin{align*}
& \mathcal{F}(w_i,T_i)  \le  \beta \int_{T_i} (1+|e(w_i)|^p)\, {\rm d}x + \beta \mathcal{H}^{d-1}(J_{w_i} \cap T_i) \\
& \le \beta  \int_{T_i} (1+|\varphi_i e(u) + (1-\varphi_i)e(v) + \nabla \varphi_i \odot (u-v) |^p) + \beta \mathcal{H}^{d-1}( (J_{u} \cup J_v) \cap T_i) \\
& \le \beta\mathcal{L}^d(T_i) + 3^{p-1}\beta  \int_{T_i} \big(|e(u)|^p + |e(v)|^p + |\nabla \varphi_i|^p |u-v|^p \big) + \beta \mathcal{H}^{d-1}( J_{u}  \cap T_i) + \beta \mathcal{H}^{d-1}( J_v \cap T_i)  \\
& \le  3^{p-1}  \beta\alpha^{-1} \big(  \mathcal{F}(u,T_i) +  \mathcal{F}(v,T_i)\big) + 3^{p-1}\beta \Vert \nabla \phi_i \Vert_\infty^p \Vert u- v\Vert^p_{L^p(T_i)} + \beta\mathcal{L}^d(T_i).
\end{align*}
 Notice that we can obtain the same estimate also  if $\mathcal{F}$ satisfies {\rm (H$_4^\prime$)},  $u \in SBD^p(A)$, and $v \in SBD^p(\B)$.  (We refer to  \cite[Proof of Proposition 3.1]{Braides-Defranceschi-Vitali} for details.)   Consequently, recalling \eqref{eq: kdef}  and using (H$_1$)  we find $i_0 \in \lbrace 1, \ldots, k \rbrace$ such that
$$ \mathcal{F}(w_{i_0},T_{i_0}) \le \frac{1}{k}\sum_{i=1}^k  \mathcal{F}(w_{i},T_i) \le \eta \big(  \mathcal{F}(u,A) +  \mathcal{F}(v,\B)\big) + M \Vert u-v \Vert^p_{L^p((A \setminus A') \cap \B)} + \eta \mathcal{L}^d( (A \setminus A') \cap \B), $$ 
where $M :=  3^{p-1}  \beta k^{-1} \max_{i=1,\ldots,k} \Vert \nabla \varphi_i \Vert_\infty^p$. This along with \eqref{eq: wi} concludes the proof of \eqref{eq: assertionfund}  by setting $w = w_{i_0}$.  To see the scaling property \eqref{eq: constant-scal}, it suffices to use the cut-off functions 
$\varphi^\eps_i\in C_0^\infty((A_{i+1})_{\eps,x_0};[0,1])$ $i=1,\ldots,k$, defined by $\varphi_i^\eps (x) =  \varphi_i( x_0 +  (x-x_0)/\eps) $ for $x \in (A_{i+1})_{\eps,x_0}$. This concludes the proof.   
 \end{proof}

\section{The global method}\label{sec: global method}

This section is devoted to the proof of Theorem \ref{theorem: PR-representation} which is based on three ingredients.  First,  we show that $\mathcal{F}$ is equivalent to $\mathbf{m}_{\mathcal{F}}$ (see \eqref{eq: general minimization}) in the sense that the two quantities have the same  Radon-Nikodym  derivative with respect to $\mu:= \mathcal{L}^d\lfloor_{\Omega} + \mathcal{H}^{d-1}\lfloor_{J_u \cap \Omega}$.

\begin{lemma}\label{lemma: G=m}
Suppose that $\mathcal{F}$ satisfies {\rm (${\rm H_1}$)}--{\rm (${\rm H_4}$)}. Let $u \in GSBD^p(\Omega)$ and   $\mu = \mathcal{L}^d\lfloor_{\Omega} + \mathcal{H}^{d-1}\lfloor_{J_u \cap \Omega}$. Then for $\mu$-a.e.\ $x_0 \in \Omega$ we have
 $$\lim_{\eps \to 0}\frac{\mathcal{F}(u,B_\eps(x_0))}{\mu(B_\eps(x_0))} =  \lim_{\eps \to 0}\frac{\mathbf{m}_{\mathcal{F}}(u,B_\eps(x_0))}{\mu(B_\eps(x_0))}.$$
\end{lemma}

We prove this lemma in the final part of this section. 
The second ingredient is that, asymptotically as $\eps \to 0$, the minimization problems $\mathbf{m}_{\mathcal{F}}(u,B_\eps(x_0))$ and $\mathbf{m}_{\mathcal{F}}(\bar{u}^{\rm bulk}_{x_0},B_\eps(x_0))$ coincide for $\mathcal{L}^{d}$-a.e.\ $x_0 \in \Omega$, where we write $\bar{u}^{\rm bulk}_{x_0} :=  \ell_{x_0,u(x_0),\nabla u(x_0)}  $ for brevity, see \eqref{eq: elastic competitor}.

\begin{lemma}\label{lemma: minsame}
Suppose that $\mathcal{F}$ satisfies {\rm (${\rm H_1}$)} and  {\rm (${\rm H_3}$)}--{\rm (${\rm H_4}$)}  and let $u \in GSBD^p(\Omega)$.  Then for $\mathcal{L}^{d}$-a.e.\ $x_0 \in \Omega$  we have
\begin{align}\label{eq: PR-proof1-bulk}
  \lim_{\eps \to 0}\frac{\mathbf{m}_{\mathcal{F}}(u,B_\eps(x_0))}{\gamma_{d}\eps^{d}} =  \limsup_{\eps \to 0}\frac{\mathbf{m}_{\mathcal{F}}(\bar{u}^{\rm bulk}_{x_0},B_\eps(x_0))}{\gamma_{d}\eps^{d}}. 
 \end{align}
\end{lemma}

We   defer the proof of Lemma \ref{lemma: minsame}  to Section \ref{sec:bulk}.  The third ingredient is that, asymptotically as $\eps \to 0$, the minimization problems $\mathbf{m}_{\mathcal{F}}(u,B_\eps(x_0))$ and $\mathbf{m}_{\mathcal{F}}(\bar{u}^{\rm surf}_{x_0},B_\eps(x_0))$ coincide for $\mathcal{H}^{d-1}$-a.e.\ $x_0 \in J_u$, where we write $\bar{u}^{\rm surf}_{x_0} := u_{x_0,u^+(x_0),u^-(x_0),\nu_u(x_0)}$ for brevity, see \eqref{eq: jump competitor}.

\begin{lemma}\label{lemma: minsame2}
Suppose that $\mathcal{F}$ satisfies {\rm (${\rm H_1}$)} and  {\rm (${\rm H_3}$)}--{\rm (${\rm H_4}$)}  and let $u \in GSBD^p(\Omega)$.  Then for $\mathcal{H}^{d-1}$-a.e.\ $x_0 \in J_u$  we have
\begin{align}\label{eq: PR-proof1}
  \lim_{\eps \to 0}\frac{\mathbf{m}_{\mathcal{F}}(u,B_\eps(x_0))}{\gamma_{d-1}\eps^{d-1}} =  \limsup_{\eps \to 0}\frac{\mathbf{m}_{\mathcal{F}}(\bar{u}^{\rm surf}_{x_0},B_\eps(x_0))}{\gamma_{d-1}\eps^{d-1}}. 
  \end{align}
\end{lemma}

We   defer the proof of Lemma \ref{lemma: minsame2} to Section \ref{sec:surf},   and now proceed  to   prove Theorem \ref{theorem: PR-representation}.

\begin{proof}[Proof of Theorem \ref{theorem: PR-representation}]
\BBB In view of the assumption (H$_4$) on $\mathcal{F}$ and of the Besicovitch derivation theorem (cf.\ \cite[Theorem~2.22]{Ambrosio-Fusco-Pallara:2000}), \EEE we need to show that for $\mathcal{L}^{d}$-a.e.\ $x_0 \in \Omega$ one has 
\begin{align}\label{eq: to show1}
\frac{\mathrm{d}\mathcal{F}(u,\cdot)}{\mathrm{d}\mathcal{L}^{d}}(x_0) = f\big(x_0,u(x_0),\nabla u(x_0)\big),
\end{align}
 where $f$ was defined in \eqref{eq:fdef}, and that for $\mathcal{H}^{d-1}$-a.e.\ $x_0 \in J_u$ one has 
\begin{align}\label{eq: to show2}
\frac{\mathrm{d}\mathcal{F}(u,\cdot)}{\mathrm{d}\mathcal{H}^{d-1}\lfloor_{J_u}}(x_0) =  g\big(x_0,u^+(x_0),u^-(x_0),\nu_u(x_0)\big),
\end{align}
 where $g$ was defined in \eqref{eq:gdef}.

   By Lemma \ref{lemma: G=m} and the fact that  $\lim_{\eps \to 0} (\gamma_{d}\eps^{d})^{-1}\mu(B_\eps(x_0))=1$ for $\mathcal{L}^{d}$-a.e.\ $x_0 \in \Omega$  we deduce 
 $$\frac{\mathrm{d}\mathcal{F}(u,\cdot)}{\mathrm{d}\mathcal{L}^{d}}(x_0)  = \lim_{\eps \to 0}\frac{\mathcal{F}(u,B_\eps(x_0))}{\mu(B_\eps(x_0))} =  \lim_{\eps \to 0}\frac{\mathbf{m}_{\mathcal{F}}(u,B_\eps(x_0))}{\mu(B_\eps(x_0))}  = \lim_{\eps \to 0}\frac{\mathbf{m}_{\mathcal{F}}(u,B_\eps(x_0))}{\gamma_{d}\eps^{d}} < \infty$$
 for $\mathcal{L}^{d}$-a.e.\ $x_0 \in \Omega$. Then, \eqref{eq: to show1} follows from \eqref{eq:fdef} and Lemma \ref{lemma: minsame}.    By Lemma \ref{lemma: G=m} and the fact that  $\lim_{\eps \to 0} (\gamma_{d-1}\eps^{d-1})^{-1}\mu(B_\eps(x_0))=1$ for $\mathcal{H}^{d-1}$-a.e.\ $x_0 \in J_u$  we deduce 
 $$\frac{\mathrm{d}\mathcal{F}(u,\cdot)}{\mathrm{d}\mathcal{H}^{d-1}\lfloor_{J_u}}(x_0)  = \lim_{\eps \to 0}\frac{\mathcal{F}(u,B_\eps(x_0))}{\mu(B_\eps(x_0))} =  \lim_{\eps \to 0}\frac{\mathbf{m}_{\mathcal{F}}(u,B_\eps(x_0))}{\mu(B_\eps(x_0))}  = \lim_{\eps \to 0}\frac{\mathbf{m}_{\mathcal{F}}(u,B_\eps(x_0))}{\gamma_{d-1}\eps^{d-1}} < \infty$$
 for $\mathcal{H}^{d-1}$-a.e.\ $x_0 \in J_u$. Now, \eqref{eq: to show2} follows   from \eqref{eq:gdef} and  Lemma \ref{lemma: minsame2}.
 \end{proof}

In the remaining part of the section we prove Lemma~\ref{lemma: G=m}. We  basically follow the lines  of \cite{BFLM, BFM, Conti-Focardi-Iurlano:15}, with the difference that the required compactness results are more delicate due to the weaker growth condition from below (see {\rm (${\rm H_4}$)}) compared to   \cite{BFLM, BFM, Conti-Focardi-Iurlano:15}. We start with some notation. For $\delta>0$ and $A \in \mathcal{A}(\Omega)$,  we define 
\begin{align}\label{eq: delta-m}
\mathbf{m}^\delta_{\mathcal{F}}(u,A) = \inf\Big\{ \sum\nolimits_{i=1}^\infty \mathbf{m}_{\mathcal{F}}(u,B_i)\colon & \ B_i \subset A   \text{ pairwise disjoint balls}, \,  \diam(B_i) \le \delta,\notag\\ 
 & \quad   \quad  \quad  \quad  \quad \quad  \quad  \quad \quad \quad \quad \quad \mu\Big(  A \setminus \bigcup\nolimits_{i=1}^\infty B_i\Big) = 0\Big\} ,
\end{align}
where, as before, $\mu= \mathcal{L}^d\lfloor_{\Omega} + \mathcal{H}^{d-1}\lfloor_{J_u \cap \Omega}$. As $\mathbf{m}^\delta_{\mathcal{F}}(u,A) $ is decreasing in $\delta$, we can also introduce
\begin{align}\label{eq: m*}
\mathbf{m}^*_{\mathcal{F}}(u,A) =   \lim_{\delta \to 0 }  \mathbf{m}^\delta_{\mathcal{F}}(u,A).
\end{align}

\BBB  In the following lemma, we prove that $\mathcal{F}$ and $\mathbf{m}^*_{\mathcal{F}}$ coincide under our assumptions.   \EEE

\begin{lemma}\label{lemma: G=m-2}
Suppose that $\mathcal{F}$ satisfies {\rm (${\rm H_1}$)}--{\rm (${\rm H_4}$)} and  let  $u \in GSBD^p(\Omega)$. Then, for all $A \in \mathcal{A}(\Omega)$ there holds $\mathcal{F}(u,A) = \mathbf{m}^*_{\mathcal{F}}(u,A)$. 
\end{lemma}

\begin{proof}
 We follow the lines of the proof of \cite[Lemma 4.1]{Conti-Focardi-Iurlano:15} focusing on the necessary adaptions due to  the weaker growth condition from below (see {\rm (${\rm H_4}$)}) compared to   \cite{Conti-Focardi-Iurlano:15}.  For each ball $B \subset A$ we have $\mathbf{m}_{\mathcal{F}}(u,B) \le \mathcal{F}(u,B)$ by definition. By {\rm (${\rm H_1}$)}  we  get $\mathbf{m}_{\mathcal{F}}^\delta(u,A) \le \mathcal{F}(u,A)$ for all $\delta>0$. This shows $\mathbf{m}_{\mathcal{F}}^*(u,A) \le \mathcal{F}(u,A)$, cf.\ \eqref{eq: m*}. 

We now address the reverse inequality.  We fix $A\in\mathcal{A}(\Omega)$ and  $\delta >0$. Let $(B^\delta_i)_i$ be balls as in the definition of $\mathbf{m}_{\mathcal{F}}^\delta(u,A)$ such that
\begin{align}\label{eq: to show-flaviana1}
\sum\nolimits_{i=1}^\infty \mathbf{m}_{\mathcal{F}}(u,B^\delta_i) \le \mathbf{m}_{\mathcal{F}}^\delta(u,A) + \delta.
\end{align}
By the definition of $\mathbf{m}_{\mathcal{F}}$, we find $v_i^\delta \in GSBD^p(B_i^\delta)$ such that $v_i^\delta = u$ in a neighborhood of $\partial B_i^\delta$ and 
\begin{align}\label{eq: new-GM2}
\mathcal{F}(v_i^\delta,B_i^\delta) \le \mathbf{m}_{\mathcal{F}}(u,B_i^\delta) + \delta \mathcal{L}^d(B_i^\delta).
\end{align}
We define 
\begin{align}\label{eq: def of vdelta}
 v^{\delta,n}  := \sum\nolimits_{i=1}^n v^\delta_i \chi_{B^\delta_i} +  u \chi_{N_0^{\delta,n}}  \quad \text{for $n\in \N$},\quad \quad \quad v^\delta := \sum\nolimits_{i=1}^\infty v^\delta_i \chi_{B^\delta_i} + u \chi_{N_0^\delta}, 
\end{align}
 where  $N_0^{\delta,n} := \Omega \setminus \bigcup_{i=1}^n B^\delta_i$   and $N_0^\delta := \Omega \setminus \bigcup_{i=1}^\infty B^\delta_i$.  
 By construction, we have that  each $ v^{\delta,n} $ lies in $GSBD^p(\Omega)$ and that 
   $\sup_{n \in \N} (\Vert e( v^{\delta,n}  ) \Vert_{L^p(\Omega)} + \mathcal{H}^{d-1}(J_{ v^{\delta,n} })) <+\infty$  by \eqref{eq: to show-flaviana1}--\eqref{eq: new-GM2} and {\rm (${\rm H_4}$)}. Moreover, $ v^{\delta,n}  \to v^\delta$ pointwise a.e.\ in $\Omega$. Then,  
\cite[Theorem~1.1]{Crismale1}   yields $v^\delta \in GSBD^p(\Omega)$. 
\begin{align}\label{eq: to show-flaviana2}
\mathcal{F}(v^\delta,A) &= \sum\nolimits_{i=1}^\infty\mathcal{F}(v_i^\delta,B_i^\delta)  + \mathcal{F}(u,N_0^\delta \cap A) \le  \sum\nolimits_{i=1}^\infty \big(\mathbf{m}_{\mathcal{F}}(u,B_i^\delta) + \delta \mathcal{L}^d(B_i^\delta)\big) \notag\\
&\le \mathbf{m}_{\mathcal{F}}^\delta(u,A) + \delta(1+\mathcal{L}^d(A)),    
\end{align}
where we also used  the fact that $\mu(N_0^\delta \cap A)  =  \mathcal{F}(u,N_0^\delta \cap A) = 0$ by the definition of $(B^\delta_i)_i$  and (${\rm H_4})$. For later purpose, we also note by (${\rm H_4}$) that this implies  
\begin{align}\label{eq: to show-flaviana2-NNN}
\Vert e(v^\delta) \Vert^p_{L^p(A)}  + \mathcal{H}^{d-1}(J_{v^\delta} \cap A) \le \alpha^{-1} \big(\mathbf{m}_{\mathcal{F}}^\delta(u,A) + \delta(1+\mathcal{L}^d(A)) \big).
\end{align}
 We now  claim that $v^\delta \to u$ in measure on $A$. To this end, we apply  Remark \ref{rem: Korn-scaling}(iii)  and Corollary~\ref{cor: kornSBDsmall}   on each $B^\delta_i$ for the function $u - v_i^\delta$ and we get sets of finite perimeter $\omega^\delta_i \subset B^\delta_i$ such that
\begin{align}\label{eq: newi}
{\rm (i)} & \ \ \big(\mathcal{L}^d(\omega^\delta_i)\big)^{(d-1)/d} \le  C    \mathcal{H}^{d-1}\big( (J_u \cup v^\delta)  \cap B^\delta_i \big),\notag\\
{\rm (ii)} & \ \   \Vert u -v^\delta_i \Vert^p_{L^{p}(B^\delta_i \setminus \omega^\delta_i)}\le  C  \delta^p \big(\Vert e(u) \Vert^p_{L^p(B^\delta_i)} + \Vert e(v^\delta) \Vert^p_{L^p(B^\delta_i)} \big),
\end{align}
 for a constant $C>0$ only depending on $p$.  Here, we used that ${\rm diam}( B_i^\delta  ) \le \delta$ and the fact that $(u - v^\delta)\lfloor_{B^\delta_i} \in GSBD^p(B^\delta_i)$ with trace zero on $\partial B^\delta_i$.  We define $\psi\colon [0,+\infty) \to [0,+\infty) $ by $\psi(t) = \min \lbrace t^p, 1\rbrace$ and observe that $v^\delta \to u$ in measure on $A$     
is equivalent to $\int_A \psi(|u-v^\delta|)\, {\rm d}x \to 0$ as $\delta \to 0$. In view of \eqref{eq: def of vdelta}, we compute
\begin{align}\label{eq: compa1}
\int_A \psi(|u-v^\delta|)\, {\rm d}x &  = \sum\nolimits_{i=1}^\infty \int_{B^\delta_i} \psi(|u-v_i^\delta|)\, {\rm d}x \le \sum\nolimits_{i=1}^\infty \Big( \Vert u -v^\delta_i \Vert^p_{L^{p}(B^\delta_i \setminus \omega^\delta_i)} + \mathcal{L}^d(\omega^\delta_i)    \Big).
\end{align}
By \eqref{eq: newi}(ii) and the fact that the balls $(B^\delta_i)_i$ are pairwise disjoint we get
\begin{align}\label{eq: compa2}
\sum\nolimits_{i=1}^\infty \Vert u -v^\delta_i \Vert^p_{L^{p}(B^\delta_i \setminus \omega^\delta_i)} \le   C \delta^p\big(\Vert e(u) \Vert^p_{L^p(A)} + \Vert e(v^\delta) \Vert^p_{L^p(A)} \big).
\end{align}
As $\omega^\delta_i \subset B^i_\delta$ and ${\rm diam}(B^\delta_i) \le \delta$, we further  get by \eqref{eq: newi}(i)
\begin{align}\label{eq: compa3}
\sum\nolimits_{i=1}^\infty \mathcal{L}^d(\omega^\delta_i) \le \gamma_d^{1/d}\delta\sum\nolimits_{i=1}^\infty \big(\mathcal{L}^d(\omega^\delta_i)\big)^{(d-1)/d}\le    \gamma_d^{1/d}  C  \delta \mathcal{H}^{d-1}\big((J_u\cup J_{v^\delta}) \cap A \big). 
\end{align}
Now, combining \eqref{eq: compa1}--\eqref{eq: compa3} and using \eqref{eq: to show-flaviana2-NNN}, we find $\int_A \psi(|u-v^\delta|)\, {\rm d}x \to 0$ as $\delta \to 0$.  With this, using (${\rm H_2}$),  \eqref{eq: m*},  and  \eqref{eq: to show-flaviana2} we get the required inequality  $\mathbf{m}_{\mathcal{F}}^*(u,A) \ge \mathcal{F}(u,A)$   in the limit as $\delta \to 0$.
This concludes the proof.
\end{proof}

\begin{proof}[Proof of Lemma \ref{lemma: G=m}.] 
\BBB The statement follows by repeating exactly the arguments in   \cite[Proofs of Lemma 4.2 and Lemma 4.3]{Conti-Focardi-Iurlano:15}. We report a sketch for the reader's convenience.

The definition of $\mathbf{m}_{\mathcal{F}}$ gives readily that for every $x_0 \in \Omega$
$$ \limsup_{\eps \to 0}\frac{\mathbf{m}_{\mathcal{F}}(u,B_\eps(x_0))}{\mu(B_\eps(x_0))} \leq \limsup_{\eps \to 0}\frac{\mathcal{F}(u,B_\eps(x_0))}{\mu(B_\eps(x_0))}.$$
The converse inequality follows by proving that, for every $t>0$, the set
\begin{equation*}
E_t:=\bigg\{x \in \Omega \colon \liminf_{\varepsilon\to 0}\frac{\mathcal{F}(u,B_{\eps}(x))- \mathbf{m}_{\mathcal{F}}(u,B_{\eps}(x))}{\mu(B_{\eps}(x))} > t \bigg\}
\end{equation*}
satisfies $\mu(E_t)=0$.  To this end, we fix $t>0$. We introduce  the family of balls (depending on $t$)
\begin{equation*}
 X^\delta:=\big\{B_\varepsilon(x) \colon \varepsilon<\delta, \, \overline{B_\varepsilon(x)} \subset \Omega, \, \mu (\partial B_\varepsilon(x))=0,\, \mathcal{F}(u, B_\varepsilon(x))>\mathbf{m}_\mathcal{F}(u, B_\varepsilon(x)) + t \mu(B_\varepsilon(x))\big\},
\end{equation*}
 and we show that 
\begin{equation*}
U^*:=\bigcap\nolimits_{\delta>0} \{x \in \Omega \colon B_\varepsilon(x) \in X^\delta \text{ for some }\varepsilon>0\}
\end{equation*}
satisfies
\begin{equation}\label{eq:referee}
E_t \subset U^*\,,\qquad\mu(U^*)=0.
\end{equation}
 Then $\mu(E_t)=0$ indeed holds true  and the proof is concluded.

We now confirm \eqref{eq:referee}.  The inclusion $E_t \subset U^*$ follows from the definition of $E_t$   which  permits to find, for any $x \in E_t$, $\varepsilon<\delta$ such that $\mathcal{F}(u, B_\varepsilon(x))>\mathbf{m}_\mathcal{F}(u, B_\varepsilon(x)) + t \mu(B_\varepsilon(x))$.  Note that, due to the  the left continuity of $\varepsilon\mapsto \mathbf{m}_\mathcal{F}(u, B_\varepsilon(x))$ (cf.\ \cite[Lemma~4.2]{Conti-Focardi-Iurlano:15}) one can also ensure the additional property $\mu(\partial B_\varepsilon(x))=0$ by slightly varying $\eps$. 
In order to prove that $\mu(U^*)=0$, one fixes a compact set $K \subset U^*$, two positive numbers $\delta < 	\eta$, and defines
\begin{equation*}
U^\eta:= \bigcup \{B_\varepsilon(x) \colon B_\varepsilon(x)  \in   X^\eta\}\,,\quad Y^\delta:=\big\{B_\varepsilon(x) \colon \varepsilon<\delta,\, \overline{B_\varepsilon(x)} \subset U^\eta \sm K, \ \mu(\partial B_\varepsilon(x))=0\big\}.
\end{equation*}
 By recalling also the definition of $U^*$, we see that  $X^\delta$ and $Y^\delta$ are fine covers of $K$ and $U^\eta \sm K$, respectively. Thus there exists countable many pairwise disjoint $B_i \in X^\delta$, $\hat{B}_j \in Y^\delta$, and a set $N$ with $\mu(N)=0$ such that 
$
U^\eta= \bigcup_{i} B_i \cup \bigcup_j \hat{B}_j \cup N.
$
In view of the assumptions (H$_1$), the definitions of $X^\delta$, $Y^\delta$ (in particular the balls have radii smaller than $\delta$) give that
\begin{equation*}
\mathcal{F}(u, U^\eta) \geq \sum\nolimits_i \mathbf{m}_\mathcal{F}(u, B_i) + \sum\nolimits_j \mathbf{m}_\mathcal{F}(u, \hat{B}_j) + t\, \mu\Big(\bigcup\nolimits_i B_i\Big) \geq \mathbf{m}_\mathcal{F}^\delta(u, U^\eta) + t\, \mu(K).
\end{equation*}
Passing to the limit in $\delta$,  \eqref{eq: m*}  and Lemma~\ref{lemma: G=m-2} imply
\begin{equation*}
\mathcal{F}(u, U^\eta) \geq \mathbf{m}_\mathcal{F}^*(u, U^\eta) + t\, \mu(K) = \mathcal{F}(u, U^\eta) + t\, \mu(K),
\end{equation*}
so that $\mu(K)=0$.  Then  $\mu(U^*)=0$ by the regularity of $\mu$. 
\EEE \end{proof}

 To conclude the proof of Theorem \ref{theorem: PR-representation}, it remains to prove Lemmas \ref{lemma: minsame} and \ref{lemma: minsame2}. This is the subject of the following two sections.

%
%
%
%
%
%
%

\section{The bulk density}\label{sec:bulk}

 This section is devoted to the proof of  Lemma~\ref{lemma: minsame}. We start by analyzing the blow-up at points with approximate  gradient.  The latter exists for  $\mathcal{L}^d$-a.e.\   point in $\Omega$ by Lemma \ref{lemma: approx-grad}.

\begin{lemma}[Blow-up at points with approximate gradient]\label{lemma: blow up}
Let $u \in GSBD^p(\Omega)$. Let $\theta \in (0,1)$. For $\mathcal{L}^{d}$-a.e.\ $x_0 \in \Omega$  there exists a  family  $u_\eps \in GSBD^p(B_\eps(x_0))$ such that
\begin{align}\label{eq: blow up-new} 
{\rm (i)} & \ \ \text{$u_\eps = u$ in a neighborhood of $\partial B_\eps(x_0)$,}   \quad \lim_{\varepsilon\to 0}\varepsilon^{-(d+1)}\Ld(\{u_\varepsilon\neq u\})=0,  \notag \\ 
{\rm (ii)} & \ \   \lim_{\eps \to 0}  \ \eps^{-(d+p)} \int_{B_{(1-\theta)\eps}(x_0)} \big|u_\eps(x) - u(x_0) - \nabla u(x_0)(x-x_0) \big|^p \, \mathrm{d}x = 0,\notag\\
{\rm (iii)} & \ \   \lim_{\eps \to 0} \ \eps^{-d} \int_{B_{\eps}(x_0)} \big| e(u_\eps)(x) - e(u)(x_0)\big|^p \, \mathrm{d}x = 0,\notag\\
{\rm (iv)} & \ \   \lim_{\eps \to 0} \  \eps^{-d}\,\mathcal{H}^{d-1}(J_{u_\eps})= 0.
\end{align}
\end{lemma}

\begin{proof}
Let $x_0\in\Omega$ be such that 
\begin{align}\label{eq: given properties}
{\rm (i)} & \ \  \lim_{\eps \to 0} \ \eps^{-d} \int_{B_{\eps}(x_0)} \big| e(u)(x) - e(u)(x_0)\big|^p \, \mathrm{d}x = 0,\notag\\
{\rm (ii)} & \ \   \lim_{\eps \to 0} \  \eps^{-d}\,\mathcal{H}^{d-1}(J_{u} \cap B_\eps(x_0))= 0,\notag \\
{\rm (iii)} & \  \  \lim_{\eps \to 0} \  \eps^{-d} \mathcal{L}^d\Big(\Big\{x \in B_\eps(x_0) \colon \,  \frac{|u(x) - u(x_0) - \nabla u(x_0)(x-x_0)|}{|x - x_0|}  > \varrho   \Big\} \Big)  = 0 \text{ for all $\varrho >0$}.
\end{align}
These properties hold for $\mathcal{L}^d$-a.e.\ $x_0 \in \Omega$ by   Lemma \ref{lemma: approx-grad} and the facts that $|e(u)|^p \in L^1(\Omega)$ and $J_u$ is 
\BBB countably $\hd$-rectifiable. \EEE
We use again the notation $\bar{u}^{\rm bulk}_{x_0}  = \ell_{x_0,u_0,\nabla u(x_0)} =   u(x_0) + \nabla u(x_0) (\cdot-x_0)$ for brevity,  see \eqref{eq: elastic competitor}.

Fix $\theta>0$.  We apply Theorem \ref{th: kornSBDsmall}  and Remark \ref{rem: Korn-scaling}(i)  for the function $u-\bar{u}_{x_0}^{\rm bulk}$ on the set $B_{(1-\theta)\eps}(x_0)$ to obtain a set of  finite perimeter $\omega_\eps \subset B_{(1-\theta)\eps}(x_0)$,  a function $v_\eps \in W^{1,p}(B_{(1-\theta)\eps}(x_0);  \R^d  )$ with $v_\eps=u-\bar{u}_{x_0}^{\rm bulk}$ in $B_{(1-\theta)\eps}(x_0)\setminus \omega_\eps$,  and an infinitesimal rigid motion  $a_\eps$ such that 
\begin{align}\label{eq: R2main-application}
{\rm (i)} & \ \ \mathcal{H}^{d-1}(\partial^* \omega_\eps) \le c\mathcal{H}^{d-1}(J_u \cap B_\eps(x_0)), \quad \quad \quad \ \ \ \ \mathcal{L}^d(\omega_\eps) \le c(\mathcal{H}^{d-1}(J_u \cap B_\eps(x_0)))^{d/(d-1)},\notag \\
{\rm (ii)} &  \ \ \Vert v_\eps - a_\eps \Vert_{L^{p} (B_{(1-\theta)\eps}(x_0))} \le c\,\eps  \Vert e(u-\bar{u}_{x_0}^{\rm bulk}) \Vert_{L^p(B_\eps(x_0))},
\notag \\
 {\rm (iii)} &  \ \ \Vert e(v_\eps) \Vert_{L^{p} (B_{(1-\theta)\eps}(x_0))} \le c\,  \Vert e(u-\bar{u}_{x_0}^{\rm bulk}) \Vert_{L^p(B_\eps(x_0))},
\end{align}
where $c>0$ depends only on $p$, cf.\ also Remark \ref{rem: Korn-scaling}(iii).  We directly note by \eqref{eq: given properties}(ii) and \eqref{eq: R2main-application}(i) that 
\begin{align}\label{eq: omegapes}
\lim_{\eps \to 0 }\eps^{ -d^2/(d-1)  }  \mathcal{L}^d(\omega_\eps) = 0. 
\end{align}
We define $u_\eps\in GSBD^p(B_\eps(x_0))$ as 
\begin{align}\label{eq: ueps-def}
u_\eps :=  u \chi_{B_\eps(x_0) \setminus B_{(1-\theta)\eps}(x_0)} + (v_\eps+\bar{u}_{x_0}^{\rm bulk})\chi_{B_{(1-\theta)\eps}(x_0)},
\end{align}
and proceed by confirming the properties stated in \eqref{eq: blow up-new}.  Notice that, by construction, $u_\eps=u$ in $B_\eps(x_0) \setminus \omega_\eps$.  First, \eqref{eq: blow up-new}(i) follows directly from the fact that $\omega_\eps \subset B_{(1-\theta)\eps}(x_0)$,   as well as \eqref{eq: omegapes}--\eqref{eq: ueps-def}. Moreover, \eqref{eq: R2main-application}(i) and \eqref{eq: given properties}(ii) imply \eqref{eq: blow up-new}(iv).  As for  \eqref{eq: blow up-new}(iii), we notice that  by  \eqref{eq: R2main-application}(iii)  and  \eqref{eq: given properties}(i) we have
\[
\lim_{\eps \to 0} \ \eps^{-d} \int_{B_{(1-\theta)\eps}(x_0)} | e(v_\eps)(x)|^p \, \mathrm{d}x = 0\,.
\] 
Since, by a direct computation, $e(u_\eps)(x)-e(u)(x_0)=e(v_\eps)(x)$ for $x \in  B_{(1-\theta)\eps}(x_0)$,  see \eqref{eq: ueps-def},  in combination with \eqref{eq: given properties}(i)  we obtain  \eqref{eq: blow up-new}(iii).   It therefore remains to prove \eqref{eq: blow up-new}(ii).

To this end, fix $\varrho>0$ and define  $\hat{\omega}_\eps := \lbrace x\in B_\eps(x_0) \colon \, |u(x) - \bar{u}^{\rm bulk}_{x_0}(x)| > \varrho \eps \rbrace$. In view of \eqref{eq: given properties}(iii) and \eqref{eq: omegapes}, we can choose  $\eps_0>0$ sufficiently small such that for all $0 <\eps \le \eps_0$ we have 
\begin{align}\label{eq: volume estimate} 
\mathcal{L}^d(\omega_\eps \cup \hat{\omega}_\eps ) \le \tfrac{1}{2}  \mathcal{L}^d(B_{(1-\theta)\eps}(x_0)).
\end{align}
By the definition of $\hat{\omega}_\eps$   and the fact that $v_\eps=u-\bar{u}_{x_0}^{\rm bulk}$ in $B_{(1-\theta)\eps}(x_0)\setminus \omega_\eps$,   we have $|v_\eps(x)|\le \varrho \eps$ for all $x \in B_{(1-\theta)\eps}(x_0) \setminus (\omega_\eps \cup \hat{\omega}_\eps)$.  Hence,  \eqref{eq: R2main-application}(ii) and the triangle inequality give
\begin{align*}
\Vert  a_\eps \Vert^p_{L^p(B_{(1-\theta)\eps}(x_0) \setminus (\omega_\eps \cup \hat{\omega}_\eps))} \le C\eps^p  \Vert e(u- \bar{u}_{x_0}^{\rm bulk}) \Vert^p_{L^p(B_\eps(x_0))} +  C \mathcal{L}^d(B_\eps(x_0)) \varrho^p \eps^p,
\end{align*}
where $C>0$ depends only on $p$. By  \eqref{eq: volume estimate} and   Lemma \ref{lemma: rigid motion} we get   
$$\Vert  a_\eps \Vert^p_{L^p(B_{(1-\theta)\eps}(x_0))} \le C\eps^p  \Vert e(u - \bar{u}_{x_0}^{\rm bulk}) \Vert^p_{L^p(B_\eps(x_0))} +  C  \mathcal{L}^d(B_\eps(x_0)) \varrho^p \eps^p.$$
Therefore, by using also \eqref{eq: given properties}(i), we derive
$
\limsup_{\eps \to 0} \eps^{-(d+p)}\Vert  a_\eps \Vert^p_{L^p(B_{(1-\theta)\eps}(x_0))}  \le  C  \gamma_d\varrho^p.
$
As $\varrho>0$ was arbitrary, we get
\begin{align}\label{eq: small pert}
\lim_{\eps \to 0} \eps^{-(d+p)}\int_{B_{(1-\theta)\eps}(x_0)}  |  a_\eps |^p \, {\rm d}x  = 0.
\end{align}
 Now, \eqref{eq: given properties}(i)  and  \eqref{eq: R2main-application}(ii)    give  that
\[
\lim_{\eps \to 0} \eps^{-(d+p)} \Vert v_\eps - a_\eps \Vert^p_{L^{p}(B_{(1-\theta)\eps}(x_0))} \le c\,\eps ^{-d} \Vert e(u-\bar{u}_{x_0}^{\rm bulk}) \Vert^p_{L^p(B_\eps(x_0))}=0\,.
\] 
As $u_\eps-\bar{u}_{x_0}^{\rm bulk}=v_\eps$ in $B_{(1-\theta)\eps}(x_0)$,  this shows \eqref{eq: blow up-new}(ii)  by  \eqref{eq: small pert}.  
\end{proof}

%
%

We are now in a position to prove Lemma \ref{lemma: minsame}.

\begin{proof}[Proof of Lemma \ref{lemma: minsame}]
It suffices  to prove  \eqref{eq: PR-proof1-bulk}  for points $x_0 \in \Omega$ where the statement of Lemma~\ref{lemma: blow up} holds  and we have $\lim_{\eps \to 0} \eps^{-d}\mu(B_\eps(x_0))=  \gamma_{d}$. This holds true for $\mathcal{L}^d$-a.e.\ $x_0 \in \Omega$. Then also $ \lim_{\eps \to 0}\eps^{-d}\mathbf{m}_{\mathcal{F}}(u,B_\eps(x_0))\in \R$ exists,  see Lemma \ref{lemma: G=m}.   As before, we write  $\bar{u}^{\rm bulk}_{x_0} = u(x_0) + \nabla u(x_0)(\cdot -x_0)$ for shorthand. 
\medskip\\
\emph{Step 1 (Inequality ``$\le$'' in \eqref{eq: PR-proof1-bulk}):}  We  fix $\eta>0$ and $\theta>0$.  Choose $z_\eps \in GSBD^p(B_{(1-3 \theta)\eps}(x_0))$ with $z_\eps = \bar{u}^{\rm bulk}_{x_0}$ in a neighborhood of $\partial B_{(1-3\theta)\eps}(x_0)$ and
\begin{align}\label{eq:repr1+}
\mathcal{F}\big(z_\eps,B_{(1-3\theta)\eps}(x_0)\big) \le \mathbf{m}_{\mathcal{F}}\big(\bar{u}^{\rm bulk}_{x_0},B_{(1-3\theta)\eps}(x_0)\big) + \eps^{d+1}.
\end{align}
We extend $z_\eps$ to a function in $GSBD^p(B_\eps(x_0))$ by setting  $z_\eps = \bar{u}^{\rm bulk}_{x_0}$ outside $B_{(1-3\theta)\eps}(x_0)$. Let $(u_\eps)_\eps$ be the  family  given by Lemma \ref{lemma: blow up}.   We  apply Lemma \ref{lemma: fundamental estimate} on $z_\eps$ (in place of $u$) and $u_\eps$ (in place of $v$) for $\eta$ as above  and the sets
\begin{align}\label{eq: definition of sets}
A' = B_{1-2\theta}(x_0), \quad  A = B_{1-\theta}(x_0), \quad  \B = B_1(x_0) \setminus  \overline{B_{1-4\theta}(x_0)}.
\end{align}
By \eqref{eq: assertionfund}--\eqref{eq: constant-scal} there exist  functions  $w_\eps \in GSBD^p(B_\eps(x_0))$ such that $w_\eps = u_\eps$ on $B_\eps(x_0) \setminus B_{(1-\theta)\eps}(x_0)$ and 
\begin{align}\label{eq:repr1++}
\mathcal{F}&(w_\eps, B_\eps(x_0)) \le   (1+\eta)\big(\mathcal{F}(z_\eps,A_{\eps,x_0})  + \mathcal{F}(u_\eps, \B_{\eps,x_0})\big) + \frac{M}{\eps^p} \Vert z_\eps - u_\eps \Vert^p_{L^p((A \setminus A')_{\eps,x_0})} + \mathcal{L}^d(B_\eps(x_0))\eta,
\end{align}
where  $M>0$ depends on $\theta$ and $\eta$, but is independent of $\eps$.  Here and in the following, we use notation \eqref{eq: shift-not}.  
In particular, we have $w_\eps = u_\eps = u$ in a neighborhood of $\partial B_\eps(x_0)$ by \eqref{eq: blow up-new}(i). By  \eqref{eq: blow up-new}(ii),   \eqref{eq: definition of sets}, and the fact that $z_\eps = \bar{u}^{\rm bulk}_{x_0}$ outside $B_{(1-3\theta)\eps}(x_0)$ we find
\begin{align}\label{eq:repr8-new}
\lim_{\eps \to 0} \ \eps^{-(d+p)} \Vert z_\eps - u_\eps \Vert^p_{L^p( (A \setminus A')_{\eps,x_0})} = \lim_{\eps \to 0} \ \eps^{-(d+p)} \Vert u_\eps - \bar{u}^{\rm bulk}_{x_0} \Vert^p_{L^p(B_{(1-\theta)\eps}(x_0))} =0.
\end{align}
This along with \eqref{eq:repr1++} shows that there exists  a sequence $(\rho_\eps)_\eps \subset (0,+\infty)$ with $\rho_\eps \to 0$ such that  
\begin{align}\label{eq:repr1+++}
\mathcal{F}&(w_\eps, B_\eps(x_0)) \le   (1+\eta)\big(\mathcal{F}(z_\eps, A_{\eps,x_0})  + \mathcal{F}(u_\eps,  \B_{\eps,x_0})\big) +  \eps^d\rho_\eps +\gamma_d\eps^d\eta.
\end{align}
On the one hand, by using that  $z_\eps = \bar{u}^{\rm bulk}_{x_0}$ on  $B_{\eps}(x_0) \setminus B_{(1-3\theta)\eps}(x_0) \subset \B_{\eps,x_0} $,  {\rm (${\rm H_1}$)},  {\rm (${\rm H_4}$)},   and \eqref{eq:repr1+} we compute  
\begin{align}\label{eq: rep1}
\limsup_{\eps\to 0}\frac{\mathcal{F}(z_\eps,A_{\eps,x_0})}{\eps^{d}} &\le \limsup_{\eps\to 0} \frac{\mathcal{F}(z_\eps,B_{(1- 3 \theta)\eps}(x_0))}{\eps^{d}}+  \limsup_{\eps\to 0} \frac{\mathcal{F}(\bar{u}^{\rm bulk}_{x_0}, \B_{\eps,x_0})}{\eps^{d}}\notag\\
&  \le  \limsup_{\eps\to 0} \frac{\mathbf{m}_{\mathcal{F}}(\bar{u}^{\rm bulk}_{x_0},B_{(1-3\theta)\eps}(x_0))}{\eps^{d}} + \beta \,  \gamma_d  \left[1-(1-4\theta)^{d}\right] (1+|e( u)(x_0)|^p) \notag \\
&\le  (1-3\theta)^{d}\limsup_{\eps'\to 0} \frac{\mathbf{m}_{\mathcal{F}}(\bar{u}^{\rm bulk}_{x_0},B_{\eps'}(x_0))}{(\eps')^{d}} +\beta \, \gamma_d\left[1-(1-4\theta)^{d}\right] (1+|e(u)(x_0)|^p),
\end{align}
where in the  last  step we substituted $(1-3\theta)\eps$ by $\eps'$. On the other hand, by  {\rm (${\rm H_4}$)} and \eqref{eq: definition of sets} we also find
\begin{align*}
&\mathcal{F}(u_\eps, \B_{\eps,x_0}) \le \beta \int_{\B_{\eps,x_0}} (1+  |e(u_\eps)|^p) + \beta \,\mathcal{H}^{d-1}(J_{u_\eps} \cap \B_{\eps,x_0}) \\
&\le \gamma_d\eps^d \beta\left[1-(1-4\theta)^{d}\right] (1+   2^{p-1}|e(u)(x_0) |^p) + 2^{p-1} \beta\Vert e(u_\eps) - e(u)(x_0) \Vert^p_{L^p(B_\eps( x_0  ))}   + \beta\mathcal{H}^{d-1}(J_{u_\eps}). 
\end{align*}
By \eqref{eq: blow up-new}(iii),(iv) this implies 
\begin{align}\label{eq: rep1XXX}
\limsup_{\eps \to 0} \frac{\mathcal{F}(u_\eps, \B_{\eps,x_0})}{\eps^{d}}  \le \beta  \, \gamma_d  \left[1-(1-4\theta)^{d}\right](1+   2^{p-1}|e(u)(x_0) |^p). 
\end{align}
Recall that  $w_\eps = u$ in a neighborhood of $\partial B_\eps(x_0)$. This along with  \eqref{eq:repr1+++}--\eqref{eq: rep1XXX} and  $\rho_\eps \to 0$ yields
\begin{align*}
 \lim\nolimits_{\eps \to 0}\frac{\mathbf{m}_{\mathcal{F}}(u,B_\eps(x_0))}{\gamma_{d}\eps^{d}} &\le  \limsup\nolimits_{\eps \to 0}\frac{\mathcal{F}(w_\eps, B_\eps(x_0))}{\gamma_{d}\eps^{d}} \notag \\
 & \le (1+\eta) \, (1-3\theta)^{d} \limsup\nolimits_{\eps \to 0}  \frac{\mathbf{m}_{\mathcal{F}}(\bar{u}^{\rm bulk}_{x_0},B_\eps(x_0))}{\gamma_{d}\eps^{d}} \\ & \ \ \  + 2(1+\eta) \beta \left[1-(1-4\theta)^{d}\right] (1+   2^{p-1}|e(u)(x_0) |^p) + \eta.
\end{align*}
Passing to $\eta,\theta \to 0$ we obtain  inequality ``$\le$'' in \eqref{eq: PR-proof1-bulk}.
\medskip\\
\emph{Step 2 (Inequality ``$\ge$'' in \eqref{eq: PR-proof1-bulk}):}  We fix $\eta,\,\theta>0$ and let   $(u_\eps)_\eps$ be again the   family  from Lemma~\ref{lemma: blow up}.    By \eqref{eq: blow up-new}(i) and Fubini's Theorem, for each $\eps >0$  we can find $ s_\eps \in (1-4\theta, 1-3\theta)\eps  $ such that
\begin{align}\label{1404201943'}
{\rm (i)}&  \ \ \lim_{\eps \to 0}\varepsilon^{-d} \, \hd\big(\{u\neq u_{\eps}\} \cap \partial B_{s_\eps}(x_0)\big)= 0,\notag\\
{\rm (ii)} & \ \ \hd\big((J_u \cup J_{u_{\eps}}) \cap  \partial B_{s_\eps}(x_0)  \big)=0 \quad \text{ for all  $\eps >0$}.
\end{align}
 We consider $z_\eps \in GSBD^p(B_{s_\eps}(x_0))$ such that $z_\eps= u$ in a neighborhood of $\partial B_{s_\eps}(x_0)$, and 
\begin{equation}\label{1404201956'}
\mathcal{F}(z_\eps,B_{s_\eps}(x_0)) \le \mathbf{m}_{\mathcal{F}}(u, B_{s_\eps}(x_0)) + \eps^{d+1}.
\end{equation} 
 We extend $z_\eps$ to a function in $GSBD^p(B_{\eps}(x_0))$ by setting 
\begin{equation}\label{1404201958'}
z_\eps=u_{\eps}\quad \text{in }B_{\eps}(x_0) \sm B_{s_\eps}(x_0).
\end{equation}
  We apply Lemma \ref{lemma: fundamental estimate} on $z_\eps$ (in place of $u$) and $\bar{u}_{x_0}^{\rm bulk}$ (in place of $v$) for   the sets indicated in \eqref{eq: definition of sets}. By   \eqref{eq: assertionfund}--\eqref{eq: constant-scal} there exist functions  $w_\eps \in  GSBD^p(B_\eps(x_0))$  such that $w_\eps = \bar{u}^{\rm bulk}_{x_0}$ on $B_\eps(x_0) \setminus B_{(1-\theta)\eps}(x_0)$ and 
$$ \mathcal{F}(w_\eps, B_\eps(x_0)) \le   (1+\eta)\big(\mathcal{F}(z_\eps,A_{\eps,x_0})  + \mathcal{F}(\bar{u}^{\rm bulk}_{x_0}, \B_{\eps,x_0})\big) +   \frac{M}{\eps^p} \Vert z_\eps - \bar{u}^{\rm bulk}_{x_0} \Vert^p_{L^p((A \setminus A')_{\eps,x_0})} + \mathcal{L}^d(B_\eps(x_0))\eta.$$
 By \eqref{1404201958'}  and  the choice of $s_\eps$ we get  that $z_\eps = u_\eps$ outside $B_{(1-3\theta)\eps}(x_0)$.  Thus,  similar to Step 1, cf.\  \eqref{eq:repr8-new} and \eqref{eq:repr1+++}, we find a sequence $(\rho_\eps)_\eps\subset (0,+\infty)$ with $\rho_\eps \to 0$ such that 
\begin{align}\label{eq: rep2-new}
\mathcal{F}&(w_\eps, B_\eps(x_0)) \le   (1+\eta)\big(\mathcal{F}(z_\eps,A_{\eps,x_0})  + \mathcal{F}(\bar{u}^{\rm bulk}_{x_0}, \B_{\eps,x_0})\big) +  \eps^d\rho_\eps +\gamma_d\eps^d\eta.
\end{align}
 Let us now estimate the  terms in \eqref{eq: rep2-new}.  We get by {\rm (${\rm H_1}$)},  {\rm (${\rm H_4}$)},  \eqref{1404201956'}--\eqref{1404201958'}, and the choice of $s_\eps$   that
\begin{equation}\label{eq: rep1-new'''}
\mathcal{F}(z_\eps,A_{\eps,x_0})
\leq \mathbf{m}_{\mathcal{F}}(u, B_{s_\eps}(x_0)) + \eps^{d+1} + \beta \, \hd\big((\{u\neq u_{\eps}\} \cup J_u \cup J_{u_{\eps}}) \cap \partial B_{s_\eps}(x_0)\big) + \mathcal{F}(u_{\eps}, \B_{\eps,x_0}).
\end{equation}
Therefore, by \eqref{eq: rep1XXX}, \eqref{1404201943'}, and the fact that $s_\eps \le (1-3\theta)\eps$  we derive 
\begin{align}\label{eq: rep1-new}
\limsup_{\eps\to 0}\frac{\mathcal{F}(z_\eps, A_{\eps,x_0})}{\eps^{d}}&\le (s_\eps/\eps)^d\limsup_{\eps\to 0} \frac{\mathbf{m}_{\mathcal{F}}(u,B_{s_\eps}(x_0))}{s_\eps^{d}} +  \beta\, \gamma_d\left[1-(1-4\theta)^{d}\right](1+   2^{p-1}|e(u)(x_0) |^p)\notag\\
&\le (1- 3  \theta)^{d}\limsup_{\eps\to 0} \frac{\mathbf{m}_{\mathcal{F}}(u,B_{\eps}(x_0))}{\eps^{d}} +  \beta\, \gamma_d\left[1-(1-4\theta)^{d}\right](1+   2^{p-1}|e(u)(x_0) |^p). 
\end{align}
  Estimating $\mathcal{F}(\bar{u}^{\rm bulk}_{x_0}, \B_{\eps,x_0})$ as in \eqref{eq: rep1}, with \eqref{eq: rep2-new}--\eqref{eq: rep1-new} and     $\rho_\eps \to 0$ we then obtain   
\begin{align*}
 \limsup_{\eps \to 0}\frac{\mathcal{F}(w_\eps, B_\eps(x_0))}{\eps^{d}} & 
  \le (1+\eta) \, (1-3\theta)^{d} \limsup\nolimits_{\eps \to 0}  \frac{\mathbf{m}_{\mathcal{F}}(u,B_\eps(x_0))}{\eps^{d}} \\ & \ \ \  + 2(1+\eta) \beta \, \gamma_d \left[1-(1-4\theta)^{d}\right] (1+   2^{p-1}|e(u)(x_0) |^p) + \gamma_d \eta.
\end{align*}
Passing to $\eta,\theta \to 0$ and recalling that  $w_\eps = \bar{u}^{\rm bulk}_{x_0}$ in a neighborhood of $\partial B_\eps(x_0)$ we derive 
$$\limsup_{\eps \to 0}\frac{\mathbf{m}_{\mathcal{F}}(\bar{u}^{\rm bulk}_{x_0},B_\eps(x_0))}{\gamma_{d}\eps^{d} }\le     \limsup_{\eps \to 0}\frac{\mathcal{F}(w_\eps, B_\eps(x_0))}{\gamma_d\eps^{d}}   \le \limsup_{\eps \to 0}  \frac{\mathbf{m}_{\mathcal{F}}(u,B_\eps(x_0))}{ \gamma_{d}\eps^{d}  } =  \lim_{\eps \to 0}  \frac{\mathbf{m}_{\mathcal{F}}(u,B_\eps(x_0))}{ \gamma_{d}\eps^{d}  }.$$
 This   shows    inequality ``$\ge$'' in \eqref{eq: PR-proof1-bulk} and concludes the proof. 
 \end{proof}

\section{The surface density}\label{sec:surf}

 This section is devoted to the proof of  Lemma~\ref{lemma: minsame2}. We start by analyzing the blow-up at jump points. 
In the following, for any $x_0 \in J_u$ 
we adopt the notation $\bar{u}_{x_0}^{\rm surf}$ for the function $u_{x_0,u^+(x_0),u^-(x_0),\nu_u(x_0)}$, see \eqref{eq: jump competitor},  with $\nu_u(x_0)\in  \mathbb{S}^{d-1}$ and $u^\pm(x_0) \in \Rd$  being  the approximate normal to $J_u$ and  the traces on both sides of $J_u$ at $x_0$, respectively. 
 Recall also notation \eqref{eq: shift-not}. 
\begin{lemma}[Blow-up at jump points]\label{le:blowupJumpPoints} 
Let $u \in GSBD^p(\Omega)$ and $\theta \in (0,1)$. For $\hd$-a.e.\ $x_0 \in J_u$ there exists a family $u_\varepsilon  \in  GSBD^p(B_\varepsilon(x_0))$ such that
\begin{align}\label{eqs:blowupJumpPoints}
{\rm (i)} & \ \ \text{$u_\eps = u$ in a neighborhood of $\partial B_\eps(x_0)$, }\quad \lim_{\varepsilon\to 0}\varepsilon^{-d}\Ld(\{u_\varepsilon\neq u\})=0, \notag \\ 
{\rm (ii)} & \ \   \lim_{\eps \to 0}  \ \eps^{-(d-1+p)} \int_{B_{(1-\theta)\eps}(x_0)} \big|u_\eps(x) - \bar{u}^{\rm surf}_{x_0}\big|^p \, \mathrm{d}x = 0,\notag\\
{\rm (iii)} & \ \  \lim_{\eps \to 0} \   \eps^{-(d-1)}\,\mathcal{H}^{d-1}(J_{u_\eps}\cap   E_{\eps,x_0}  )= \hd(\Pi_0 \cap E) \quad \text{for all Borel sets } E \subset B_1(x_0),\notag\\
{\rm (iv)} & \ \   \lim_{\eps \to 0} \ \eps^{-(d-1)} \int_{B_{\eps}(x_0)} \big| e(u_\eps)(x)\big|^p \, \mathrm{d}x = 0,
\end{align}
where $\Pi_0$ denotes the hyperplane passing  through  $x_0$ with normal $\nu_u(x_0)$.
\end{lemma}
\begin{proof}
We start by using the fact that  $J_u$ is 
\BBB countably $\hd$-rectifiable \EEE
and the blow-up properties of $GSBD^p$ functions. Arguing as in, e.g., \cite[Proof of Theorem~2]{Cha04}, \cite[Proof of Theorem~1.1]{Crismale2}, \cite[Lemma~3.4]{Conti-Focardi-Iurlano:15}, we infer  that for $\hd$-a.e.\ $x_0 \in J_u$ there exist $\ove\varepsilon>0$,  $\nu_u(x_0) \in \mathbb{S}^{d-1}$,  $u^\pm(x_0) \in \Rd$,  and   a hypersurface $\Gamma$ which is a graph of a  function  $h$  defined on $\Pi_0$,  being  $C^1$ and Lipschitz, 
 such that $x_0 \in \Gamma$, $\Pi_0$ is tangent to $\Gamma$ in $x_0$,  $\Gamma\cap B_\varepsilon(x_0)$ separates $B_\varepsilon(x_0)$ in two open connected components $B_\varepsilon^{\Gamma,\pm}(x_0)$ for  each  $\varepsilon< \ove \varepsilon$, and  
\begin{align}\label{1104201859}
{\rm (i)} & \ \ \lim_{\varepsilon\to 0} \varepsilon^{-(d-1)} \hd((J_u \triangle \Gamma) \cap B_\varepsilon(x_0))=0, \notag \\
& \ \   \lim_{\varepsilon\to 0} \varepsilon^{-(d-1)} \hd(\Gamma\cap  E_{\eps,x_0}  )= \hd(\Pi_0 \cap E) \quad \text{for all Borel sets } E \subset B_1(x_0),  \notag \\
{\rm (ii)} & \ \ \lim_{\varepsilon\to 0} \varepsilon^{-(d-1)} \int_{B_\varepsilon(x_0)} |e(u)|^p \dx =0, \notag \\
{\rm (iii)} & \ \ \lim_{\varepsilon\to 0} \varepsilon^{-d}  \Ld\big( \{x \in B_\varepsilon(x_0) \colon |u- \bar{u}^{\rm surf}_{x_0}|> \varrho\}  \big)=0\quad\text{for all }\varrho>0.
\end{align}
In particular,  (ii) follows from the fact that $|e(u)|^p \in L^1(\Omega)$ and (iii) from  \eqref{0106172148}. 
 Then,  since $\Pi_0$ is tangent to $\Gamma$ in $x_0$, $\Gamma \cap B_\varepsilon(x_0)$ is the graph of a Lipschitz function    $h_\eps$  defined on a subset of $\Pi_0$, with Lipschitz constant $L_\varepsilon$ such that $\lim_{\varepsilon\to 0} L_\varepsilon=0$.  Therefore,   it holds that  
 \begin{equation}\label{1304200828}
 \lim_{\varepsilon\to 0}\varepsilon^{-d}\Ld\big(B_\varepsilon^{\Gamma,\pm}(x_0) \triangle B_\varepsilon^\pm(x_0)\big)=0,
 \end{equation}
where $B_\varepsilon^\pm(x_0) :=\{ y \in B_\varepsilon(x_0)\colon \pm  (y-x_0)  \cdot \nu_u(x_0) >0\}$.  By this and Fubini's Theorem, for each $\eps >0$  we can find $ s_\eps \in (1-\theta, 1-\frac{\theta}{2})\eps  $ such that
 \begin{equation}\label{1304200828-new}
 \lim_{\varepsilon\to 0}\varepsilon^{-(d-1)}\Ld\Big( \big(B_\varepsilon^{\Gamma,\pm}(x_0) \triangle B_\varepsilon^\pm(x_0)\big)  \cap \partial B_{s_\eps}(x_0) \Big) =0.
 \end{equation} 
For any $\varepsilon>0$, we apply Theorem~\ref{th: kornSBDsmall}  and Remark \ref{rem: Korn-scaling}(i)  on $u$   in the two connected components $B_{s_\varepsilon}^{\Gamma,\pm}(x_0)$ for $\varepsilon< \ove \varepsilon$. 
This gives
two functions $v_\varepsilon^\pm \in W^{1,p}(B_{s_\varepsilon}^{\Gamma,\pm}(x_0);\Rd)$,  two sets of finite perimeter $\omega^\pm_\eps \subset B_{s_\eps}^{\Gamma,\pm}(x_0)$, and two infinitesimal rigid motions  $a^\pm_\eps$ such that 
\begin{align}\label{1104201900}
{\rm (i)} & \ \ v_\varepsilon^\pm=u \quad\text{in }B_{s_\varepsilon}^{\Gamma,\pm}(x_0) \sm \omega_\varepsilon^\pm,\notag \\
{\rm (ii)} & \ \ \mathcal{H}^{d-1}(\partial^* \omega^\pm_\eps) \le c\mathcal{H}^{d-1}(J_u \cap B_\eps^{\Gamma,\pm}(x_0)), \quad   \ \  \mathcal{L}^d(\omega^\pm_\eps) \le c\big(\mathcal{H}^{d-1}(J_u \cap B_\eps^{\Gamma,\pm}(x_0))\big)^{d/(d-1)},\notag \\
{\rm (iii)} &  \ \ \Vert v_\varepsilon^\pm - a^\pm_\eps \Vert_{L^{p}(B_{s_\eps}^{\Gamma,\pm}(x_0))} \le c\,\eps  \Vert e(u) \Vert_{L^p(B_\eps^{\Gamma,\pm}(x_0))}, \notag \\
{\rm (iv)} &  \ \ \Vert \nabla v_\varepsilon^\pm  - \nabla a^\pm_\eps \Vert_{L^{p}(B_{s_\eps}^{\Gamma,\pm}(x_0))} \le c   \Vert e(u) \Vert_{L^p(B_\eps^{\Gamma,\pm}(x_0))}, 
\end{align}
where $c>0$  is  independent of $\eps$.  (See Remark~\ref{rem: Korn-scaling}(ii),(iii)  and recall that the Lipschitz constant of  $h_\varepsilon$ vanishes as $\varepsilon\to 0$.)  By the Sobolev extension theorem we extend $v_\eps^\pm$ to $\hat{v}^\pm_\eps \in W^{1,p}(B_{s_\eps}(x_0);\R^d)$, and \eqref{1104201900}(iii),(iv) along with the linearity of the extension operator yield
\begin{align}\label{eq: 1104201900-new}
\eps^{-1}\Vert \hat{v}_\varepsilon^\pm - a^\pm_\eps \Vert_{L^{p}(B_{s_\eps}(x_0))} + \Vert \nabla \hat{v}_\varepsilon^\pm  - \nabla a^\pm_\eps \Vert_{L^{p}(B_{s_\eps}(x_0))} \le c   \Vert e(u) \Vert_{L^p(B_\eps^{\Gamma,\pm}(x_0))},
\end{align}
where, as before, the constant is independent of $\eps$. (Here, we used  again  the properties of the functions $h_\varepsilon$ recalled below \eqref{1104201859}.)  We define $u_\varepsilon \in GSBD^p(B_\varepsilon(x_0))$ as
\begin{equation}\label{1104202008}
u_\eps :=  \begin{dcases}
\hat{v}_\varepsilon^+ &\quad\text{in }  B_{s_\varepsilon}^+(x_0),\\
\hat{v}_\varepsilon^- &\quad\text{in } B_{s_\varepsilon}^-(x_0),\\
u  &\quad\text{in } B_\varepsilon(x_0) \sm  B_{s_\varepsilon}(x_0),
\end{dcases}
\end{equation}
where $B_{s_\varepsilon}^\pm(x_0)$ is defined below \eqref{1304200828}. We now   prove the properties in \eqref{eqs:blowupJumpPoints}.  First, by definition  we have that $u_\varepsilon=u$ in a neighborhood of $\partial B_\varepsilon(x_0)$.  By $B_{(1-\theta)\varepsilon}^{\Gamma,\pm}(x_0) \cap \Gamma = \emptyset$, \eqref{1104201859}(i), \eqref{1304200828}, and  \eqref{1104201900}(i),(ii) 
 we obtain $\lim_{\varepsilon\to 0} \varepsilon^{-d} \Ld(\{u_\varepsilon\neq u\})=0$.  This  concludes \eqref{eqs:blowupJumpPoints}(i).  Moreover, \eqref{1104201859}(ii)  and \eqref{eq: 1104201900-new}  imply \eqref{eqs:blowupJumpPoints}(iv). By definition of $u_\varepsilon$ and  \eqref{1104201900}(i) it holds that
 $$J_{u_\eps}  \subset \big(\Pi_0 \cap \overline{B_{s_\eps}(x_0)}\big)\cup \big(J_u \cap (B_\eps(x_0)\sm B_{s_\eps}(x_0)) \big) \EEE \cup   \partial^* \omega_\eps^+ \cup \partial^* \omega_\eps^- \cup \big(\big(B_\varepsilon^{\Gamma,\pm}(x_0) \triangle B_\varepsilon^\pm(x_0)\big)  \cap \partial B_{s_\eps}(x_0) \big).$$  
 We now show \eqref{eqs:blowupJumpPoints}(iii). Concerning the ''$\le$'' inequality, for a fixed  Borel  set   $E \subset B_1(x_0)$
we have to estimate the measure of the intersection with $E_{\eps,x_0}$ and any of the five sets in the right-hand side above: it holds that
$\hd\big(\Pi_0 \cap \overline{B_{s_\eps}(x_0)}\cap E_{\eps,x_0}\big)= \eps^{d-1} \EEE \hd(\Pi_0 \cap \overline{B_{{s_\eps} / \eps}(x_0)} \cap E)$ for any $\eps>0$ by rescaling, that $\lim_{\varepsilon\to 0} |\eps^{-(d-1)} \hd\big(J_u \cap (E_{\eps,x_0}\sm B_{s_\eps}(x_0)) \big)- \hd\big(\Pi_0 \cap (E\sm  B_{{s_\eps} / \eps}(x_0))\big)|=0$ by \eqref{1104201859}(i), while the last three terms are estimated by  \eqref{1104201900}(ii) and \eqref{1304200828-new}. \EEE
 To see the converse, we first apply \cite[Theorem~1.1]{Crismale1}
  to the functions
$u_\eps(x_0+\eps\cdot)$, which converge  in measure  to $ \bar{u}^{\rm surf}_{x_0}$ in $B_1(0)$ by \eqref{1104201859},  \eqref{1304200828},  and \eqref{1104201900}(ii). Then we scale back to $B_\eps(x_0)$. Hence, \eqref{eqs:blowupJumpPoints}(iii) holds.

 It remains to prove  \eqref{eqs:blowupJumpPoints}(ii). We notice that this easily follows from
\begin{equation}\label{1304201020}
\lim_{\varepsilon\to 0} \varepsilon^{-(d-1+p)} \int_{ B_{s_\eps}^{\pm}(x_0)} |a_\varepsilon^\pm- \bar{u}_{x_0}^{\rm surf}|^p \dx =0.
\end{equation}
In fact, \eqref{1104201859}(ii) and  \eqref{eq: 1104201900-new}   give that
\begin{equation}\label{1304201021}
\lim_{\varepsilon\to 0} \varepsilon^{-(d-1+p)} \int_{ B_{s_\eps}^{\pm}(x_0)} |  \hat{v}_\varepsilon^\pm   -a_\varepsilon^\pm|^p \dx =0.
\end{equation}
Then,  \eqref{1304201020}, \eqref{1304201021}, the triangle inequality,  $s_\eps \ge (1-\theta)\eps$,  and  \eqref{1104202008}  imply \eqref{eqs:blowupJumpPoints}(ii).

Therefore,  let us now  confirm \eqref{1304201020}.  We only address the ``+'' case, for  the ``-'' case is  analogous.  We first observe that by \eqref{1104201859}(iii), \eqref{1304200828}, and a diagonal argument, we may find a sequence $(\varrho_\varepsilon)_\varepsilon \subset (0,+\infty)$ with $\lim_{\varepsilon\to 0} \varrho_\varepsilon=0$ such that the sets
\[
\hat{\omega}^+_\varepsilon:=\{x \in B_\varepsilon(x_0) \colon |u(x) -   u^+(x_0)  |> \varrho_\varepsilon\} \cap B_{ s_\eps}^{\Gamma,+}(x_0)
\]
satisfy
\begin{equation}\label{1304200830}
\lim_{\varepsilon\to 0} \varepsilon^{-d}\Ld(\hat{\omega}^+_\varepsilon)=0.
\end{equation}
In view of \eqref{1104201900}(i),(iii), we have that
\begin{equation}\label{1304200933}
\Vert u - a^+_\eps \Vert_{L^{p}(B_{ s_\eps}^{\Gamma,+}(x_0)\sm \omega_\varepsilon^+)} \le c\,\eps  \Vert e(u) \Vert_{L^p(B_{ \eps }^{\Gamma,+}(x_0))}.
\end{equation}
Then, by \eqref{1304200933}, the definition of $\hat{\omega}^+_\varepsilon$, and  the  triangle inequality we get that
\begin{align}\label{1304200934}
\Vert u^+(x_0) - a^+_\eps \Vert_{L^{p}(B_{ s_\eps}^{\Gamma,+}(x_0)\sm(\omega_\varepsilon^+ \cup \hat{\omega}_\varepsilon^+))} \le c\,\eps  \Vert e(u) \Vert_{L^p(B_{ \eps}^{\Gamma,+}(x_0))} +  \gamma_{d}^{\frac{1}{p}} \varepsilon^{\frac{d}{p}} \varrho_\varepsilon. 
\end{align}
  By \eqref{1104201859}(i),  \eqref{1304200828},  \eqref{1104201900}(ii), and \eqref{1304200830}  we obtain  $\mathcal{L}^d(\omega_\eps^+ \cup \hat{\omega}^+_\eps ) \le \tfrac{1}{2}  \mathcal{L}^d(B_{ s_\eps}^{\Gamma,+}(x_0))$  for $\eps$ sufficiently small.    Then, by  Lemma \ref{lemma: rigid motion}  we have that $ \gamma_d \varepsilon^{d/p}   \Vert u^+(x_0) - a^+_\eps \Vert_{L^{\infty}(B_{\eps}(x_0))}$ is less or equal than the right-hand side of \eqref{1304200934},  up to multiplication with a constant.   This along  with \eqref{1104201859}(ii),   $p \ge 1$, and the fact that $\varrho_\eps \to 0$ 
  implies 
\begin{equation}\label{1304201009}
\lim_{\varepsilon\to 0}\Vert u^+(x_0) - a^+_\eps \Vert_{L^{\infty}(B_{\eps}(x_0))} =0.
\end{equation}
Let us consider  $A^+_\varepsilon \in \Mddskew$  and  $b^+_\varepsilon \in \Rd$  such that 
 $a_\varepsilon^+(x)= A^+_\varepsilon (x-x_0)  + b^+_\varepsilon$.  
Then \eqref{1304201020}   follows by
\begin{subequations}\label{eqs:1304201143}
\begin{equation}\label{1304201144}
\lim_{\varepsilon\to 0}  \varepsilon|A^+_\varepsilon|^p=0,
\end{equation}
\begin{equation}\label{1304201145}
\lim_{\varepsilon\to 0} \varepsilon^{\frac{1-p}{p}} |b^+_\varepsilon - u^+(x_0)|=0.
\end{equation}
\end{subequations}
 So we are left to prove  \eqref{eqs:1304201143} which  corresponds to \cite[equations (3.18)-(3.19)]{Conti-Focardi-Iurlano:15}. The proof goes in the same way with slight modifications  that  we  indicate  below. For fixed $\delta>0$ small, by \eqref{1104201859}(ii) there exists $\widehat{\varepsilon}>0$, depending on $\delta$, such that
\begin{equation}\label{1304201655}
 \varepsilon^{-(d-1)} \int_{B_\varepsilon(x_0)} |e(u)|^p \dx  \le  \delta^p \quad\text{for  all  }\varepsilon \leq \widehat{\varepsilon}. 
\end{equation}
For $\tilde{\varepsilon}< \varepsilon < \widehat{\varepsilon}$, we set $\varepsilon_k:=   \min  \lbrace 2^k \tilde{\varepsilon}, \varepsilon \rbrace$ and adopt the notation $k$ in place of $\varepsilon_k$ in the  subscripts. 
We then obtain  
\begin{equation}\label{1304201701}
\begin{split}
\| a_k^+-a_{k+1}^+  \|_{L^\infty(B_{\varepsilon_k}^+(x_0))} &  \leq     c \gamma_d^{-\frac{1}{p}}    
\varepsilon_k^{-\frac{d}{p}}\|a^+_k-a^+_{k+1}\|_{L^p(B_{ s_{\eps_k}}^{\Gamma,+}(x_0)\sm (\omega_k^+\cup \omega_{ k+1}^+))} \leq 
c\,\delta \, \varepsilon_k^{\frac{p-1}{p}}.
\end{split}
\end{equation}
In fact, the first inequality follows from  \eqref{1104201900}(ii) 
and  Lemma \ref{lemma: rigid motion}, and the  second  one from \eqref{1304200933}, \eqref{1304201655},  and the triangle inequality. 
Similarly, employing \eqref{1104201900}(i),(iv) in place of  \eqref{1304200933}, and recalling $\nabla a^+_{\eps_k} = A^+_k$   we obtain
\begin{equation}\label{1304201702}
 | A_k^+-A_{k+1}^+ |   \leq c\,
\delta \, \varepsilon_k^{-\frac{1}{p}}.
\end{equation}
At this stage, \eqref{eqs:1304201143} follow exactly as in \cite{Conti-Focardi-Iurlano:15}: for $\widehat{k}$  being  the first index such that $\varepsilon_{ \widehat{k}} = \varepsilon$,  recalling $\tilde{\eps} = \eps_0$ and  summing \eqref{1304201702} 
gives $\tilde{\varepsilon} |A^+_{\tilde{\varepsilon}}|^p \leq \tilde{\varepsilon} ( |A^+_{\widehat{k}}| + \sum_{k=0}^{\widehat{k}-1} |A^+_k - A^+_{k+1}|)^p \leq c\, \delta^p + c\, \tilde{\varepsilon} |A^+_{\widehat{k}}|^p$. The right-hand side 
vanishes as $\tilde{\varepsilon} \to 0$ and $\delta\to 0$, and this proves \eqref{1304201144}. Moreover, summing \eqref{1304201701} (and since $|b^+_k-b^+_{k+1}| \leq  \| a_k^+-a_{k+1}^+  \|_{L^\infty(B_{\varepsilon_k}^+(x_0))}$)  we  obtain   $ |b^+_{\tilde{\varepsilon}}- b^+_{\varepsilon}|  \leq c \,\delta \,\varepsilon^{\frac{p-1}{p}}$  for all $0 <\tilde{\eps} \le \eps$.  By   passing to the limit as $\tilde{\varepsilon}\to 0$ together with \eqref{1304201009},  we  get $\varepsilon^{\frac{1-p}{p}} |u^+(x_0) - b^+_\varepsilon| \leq c\,\delta$ for $\varepsilon < \widehat{\varepsilon}=\widehat{\varepsilon}(\delta)$.  Thus,   \eqref{1304201145}  follows  by the arbitrariness of $\delta>0$, concluding the proof. 
\end{proof}

\begin{remark}[Construction of $u_\eps$]\label{rem: cons}
{\normalfont
We point out that our definition of $u_\eps$ in \eqref{1104202008} differs from the corresponding constructions in \cite[Lemma 3]{BFLM} and  \cite[Lemma 3.4]{Conti-Focardi-Iurlano:15} in order  to fix a possible flaw contained in these proofs. Roughly speaking, in our notation, in \cite{BFLM, Conti-Focardi-Iurlano:15}, $u_\eps$ on $B_{ s_\varepsilon}(x_0)$ is defined as 
\begin{equation}\label{1104202008-NNN}
u_\eps =  \begin{dcases}
 {v}_\varepsilon^+  &\quad\text{in }  B_{ s_\varepsilon}^{\Gamma,+}(x_0),\\
 {v}_\varepsilon^-  &\quad\text{in } B_{ s_\varepsilon}^{\Gamma,-}(x_0).
\end{dcases}
\end{equation}
Then, instead of \eqref{1304201020} one needs  to check $\lim_{\varepsilon\to 0} \varepsilon^{-(d-1+p)} \int_{B_{s_\eps}^{\Gamma\pm}(x_0)} |a_\varepsilon^\pm- \bar{u}_{x_0}^{\rm surf}|^p \dx =0$. This, however, is in general false if  $\liminf_{\varepsilon\to 0} \varepsilon^{-(d-1+p)}\Ld\big(B_\varepsilon^{\Gamma,\pm}(x_0) \triangle B_\varepsilon^\pm(x_0)\big)>0$ (which is clearly possible). Let us also remark
that, in contrast to our construction, \eqref{1104202008-NNN} allows to prove an estimate of the form 
\begin{align}\label{eq: good asymtotici}  
 \lim_{\eps\to 0}  \eps^{-(d-1)}\mathcal{H}^{d-1}(J_{u_\eps} \setminus  J_u) = 0,
\end{align}
see \cite[Equation (24)]{BFLM} and  \cite[Lemma 3.4(i)]{Conti-Focardi-Iurlano:15}. It is not clear to us if it is possible that for $u_\eps$ satisfying  the fundamental blow-up property \eqref{eqs:blowupJumpPoints}(ii)  one may still have an estimate  of the form \eqref{eq: good asymtotici}. The latter, however, is not needed for our proofs.
}     
\end{remark}

We now proceed  with  the proof of Lemma~\ref{lemma: minsame2}.
\begin{proof}[Proof of Lemma~\ref{lemma: minsame2}]
The proof follows the same strategy of the proof of Lemma~\ref{lemma: minsame}. We fix $x_0 \in J_u$ such that the statement of Lemma~\ref{le:blowupJumpPoints} holds at $x_0$ and $\lim_{\eps \to 0} \eps^{-(d-1)}\mu(B_\eps(x_0))=  \gamma_{d-1}  $. This is possible for $\hd$-a.e.\ $x_0 \in J_u$. Then also $ \lim_{\eps \to 0}\eps^{-(d-1)}\mathbf{m}_{\mathcal{F}}(u,B_\eps(x_0))\in \R$ exists,  see Lemma \ref{lemma: G=m}.    We prove \eqref{eq: PR-proof1} for $x_0$ of this type. 
\medskip\\
\textit{Step~1 (Inequality ``$\leq$'' in \eqref{eq: PR-proof1}):}
We  fix $\eta,\,\theta>0$, and consider $z_\eps \in GSBD^p(B_{(1-3 \theta)\eps}(x_0))$ with $z_\eps = \bar{u}^{\rm surf}_{x_0}$ in a neighborhood of $\partial B_{(1-3\theta)\eps}(x_0)$ and
\begin{align}\label{eq:repr1+'}
\mathcal{F}\big(z_\eps,B_{(1-3\theta)\eps}(x_0)\big) \le \mathbf{m}_{\mathcal{F}}\big(\bar{u}^{\rm surf}_{x_0},B_{(1-3\theta)\eps}(x_0)\big) + \eps^d.
\end{align}
We extend $z_\eps$ to a function in $GSBD^p(B_\eps(x_0))$ by setting  $z_\eps = \bar{u}^{\rm surf}_{x_0}$ outside $B_{(1-3\theta)\eps}(x_0)$. Let $(u_\eps)_\eps$ be the  family  given by Lemma~\ref{le:blowupJumpPoints}.   As in the proof of Lemma~\ref{lemma: minsame}, we  apply Lemma~\ref{lemma: fundamental estimate} on $z_\eps$ (in place of $u$) and $u_\eps$ (in place of $v$) for $\eta$ fixed above  and the sets
\begin{align*}
A' = B_{1-2\theta}(x_0), \quad  A = B_{1-\theta}(x_0), \quad  \B = B_1(x_0) \setminus  \overline{B_{1-4\theta}(x_0)}.
\end{align*}
Recalling  notation \eqref{eq: shift-not} and   \eqref{eq: assertionfund}--\eqref{eq: constant-scal}, we find
$w_\eps \in GSBD^p(B_\eps(x_0))$ such that $w_\eps = u_\eps$ on $B_\eps(x_0) \setminus B_{(1-\theta)\eps}(x_0)$ and 
\begin{align}\label{eq:repr1++'}
\mathcal{F}&(w_\eps, B_\eps(x_0)) \le   (1+\eta)\big(\mathcal{F}(z_\eps,A_{\eps,x_0})  + \mathcal{F}(u_\eps, \B_{\eps,x_0})\big) + \frac{M}{\eps^p} \Vert z_\eps - u_\eps \Vert^p_{L^p((A \setminus A')_{\eps,x_0})} + \mathcal{L}^d(B_\eps(x_0))\eta,
\end{align}
where  $M>0$ depends on $\theta$ and $\eta$, but is independent of $\eps$. In particular, we have $w_\eps = u_\eps = u$ in a neighborhood of $\partial B_\eps(x_0)$ by \eqref{eqs:blowupJumpPoints}(i). By  \eqref{eqs:blowupJumpPoints}(ii) and the fact that $z_\eps = \bar{u}^{\rm surf}_{x_0}$ outside $B_{(1-3\theta)\eps}(x_0)$ we find
\begin{align*}
\lim_{\eps \to 0} \ \eps^{-(d-1+p)} \Vert z_\eps - u_\eps \Vert^p_{L^p((A \setminus A')_{\eps,x_0})} = \lim_{\eps \to 0} \ \eps^{-(d-1+p)} \Vert u_\eps - \bar{u}^{\rm surf}_{x_0} \Vert^p_{L^p(B_{(1-\theta)\eps}(x_0))} =0.
\end{align*}
Inserting this in \eqref{eq:repr1++'}, we find that, for a suitable  sequence  $(\rho_\eps)_\eps \subset (0,+\infty)$ with $\rho_\eps \to 0$, 
\begin{align}\label{eq:repr1+++'}
\mathcal{F}&(w_\eps, B_\eps(x_0)) \le   (1+\eta)\big(\mathcal{F}(z_\eps,A_{\eps,x_0})  + \mathcal{F}(u_\eps,  \B_{\eps,x_0})\big) +  \eps^{d-1}\rho_\eps +  \gamma_d  \eps^d\eta.
\end{align}
We now evaluate the first terms in the right-hand side of \eqref{eq:repr1+++'}: since  $z_\eps = \bar{u}^{\rm bulk}_{x_0}$ on  $B_{\eps}(x_0) \setminus B_{(1-3\theta)\eps}(x_0) \subset \B_{\eps,x_0} $,  by  {\rm (${\rm H_1}$)},  {\rm (${\rm H_4}$)},  and   \eqref{eq:repr1+'} we have 
\begin{align}\label{eq: rep1'}
\limsup_{\eps\to 0}\frac{\mathcal{F}(z_\eps, A_{\eps,x_0})}{\eps^{d-1}} &\le \limsup_{\eps\to 0} \frac{\mathcal{F}(z_\eps,B_{(1- 3 \theta)\eps}(x_0))}{ \eps^{d-1}}+  \limsup_{\eps\to 0} \frac{\mathcal{F}(\bar{u}^{\rm surf}_{x_0},\B_{\eps,x_0})}{ \eps^{d-1}}\notag\\
&  \le  \limsup_{\eps\to 0} \frac{\mathbf{m}_{\mathcal{F}}(\bar{u}^{\rm surf}_{x_0},B_{(1-3\theta)\eps}(x_0))}{\eps^{d-1}} + \beta  \,  \hd(\B \cap \Pi_0)\notag \\
&\le  (1-3\theta)^{d-1}\limsup_{\eps\to 0} \frac{\mathbf{m}_{\mathcal{F}}(\bar{u}^{\rm surf}_{x_0},B_{\eps}(x_0))}{\eps^{d-1}} +\beta\, \gamma_{d-1}  \big( 1-(1-4\theta)^{d-1}  \big),
\end{align}
 where,  as in Lemma~\ref{le:blowupJumpPoints}, we  denote by  $\Pi_0$ the hyperplane passing  through  $x_0$ with normal $\nu_u(x_0)$.
By  {\rm (${\rm H_4}$)} and  \eqref{eqs:blowupJumpPoints}(iii),(iv)
 we  get
\begin{align}\label{eq: rep1XXX'}
\limsup_{\eps \to 0} \frac{\mathcal{F}(u_\eps,  \B_{\eps,x_0})}{  \beta   \eps^{d-1}}  &\le  \limsup_{\eps \to 0} \frac{\int_{  \B_{\eps,x_0}} (1+  |e(u_\eps)|^p) \dx }{ \eps^{d-1}} 
+ \limsup_{\eps \to 0} \frac{\mathcal{H}^{d-1}(J_{u_\eps}\cap  \B_{\eps,x_0}) }{ \eps^{d-1}} 
\notag \\&=   \gamma_{d-1}\big( 1-(1-4\theta)^{d-1}  \big).
\end{align}
Collecting \eqref{eq:repr1+++'}, \eqref{eq: rep1'}, \eqref{eq: rep1XXX'},  and recalling  $\rho_\eps \to 0$, as well as  the fact that  
$w_\eps = u$ in a neighborhood of $\partial B_\eps(x_0)$, we obtain
\begin{equation*}
\begin{split}
 \lim_{\eps \to 0}\frac{\mathbf{m}_{\mathcal{F}}(u,B_\eps(x_0))}{ \gamma_{d-1}\eps^{d-1}} &\le  \limsup_{\eps \to 0}\frac{\mathcal{F}(w_\eps, B_\eps(x_0))}{ \gamma_{d-1}\eps^{d-1}}  \\
 & \le (1+\eta)  (1-3\theta)^{d-1} \limsup_{\eps \to 0}  \frac{\mathbf{m}_{\mathcal{F}}(\bar{u}^{\rm surf}_{x_0},B_\eps(x_0))}{ \gamma_{d-1}\eps^{d-1}}    +  2 (1+\eta) \beta\, \big( 1-(1-4\theta)^{d-1} \big).
 \end{split}
\end{equation*}
Passing to the limit as $\eta,\theta \to 0$ we conclude  inequality ``$\le$''   in \eqref{eq: PR-proof1}.  
\medskip \\
\textit{Step~2 (Inequality ``$\geq$'' in \eqref{eq: PR-proof1}):} We fix $\eta,\,\theta>0$ and let, as in Step~1, $(u_\eps)_\eps$ be the  family  given
by Lemma \ref{le:blowupJumpPoints}.   By \eqref{eqs:blowupJumpPoints}(i) and  Fubini's  Theorem,  for each $\eps >0$  we can find $ s_\eps \in (1-4\theta, 1-3\theta)\eps  $ such that
\begin{align}\label{1404201943}
{\rm (i)}&  \ \ \lim_{\eps \to 0}\varepsilon^{-(d-1)} \hd\big(\{u\neq u_{\eps}\} \cap \partial B_{s_\eps}(x_0)\big)= 0,\notag\\
{\rm (ii)} & \ \ \hd\big((J_u \cup J_{u_{\eps}}) \cap \partial B_{s_\eps}(x_0) \big)=0 \quad \text{ for all  $\eps >0$}.
\end{align}
 We consider $z_\eps \in GSBD^p(B_{s_\eps}(x_0))$ such that $z_\eps= u$ in a neighborhood of $\partial B_{s_\eps}(x_0)$,  and 
\begin{equation}\label{1404201956}
\mathcal{F}(z_\eps,B_{s_\eps}(x_0)) \le \mathbf{m}_{\mathcal{F}}(u, B_{s_\eps}(x_0)) + \eps^{d}.
\end{equation} 
 We extend $z_\eps$ to a function in $GSBD^p(B_{\eps}(x_0))$ by setting 
\begin{equation}\label{1404201958}
z_\eps=u_{\eps}\quad \text{in }B_{\eps}(x_0) \sm B_{s_\eps}(x_0).
\end{equation}
We apply Lemma \ref{lemma: fundamental estimate} for $z_{\eps}$ (in place of $u$), $\bar{u}_{x_0}^{\rm surf}$ (in place of $v$), and for the sets $A$, $A'$, $B$ as in Step~1, in correspondence to $\eps$.
By   \eqref{eq: assertionfund}--\eqref{eq: constant-scal},  there  exists 
$w_{\eps} \in GSBD^p(B_{\eps}(x_0))$ such that $w_{\eps} = \bar{u}^{\rm surf}_{x_0}$ on $B_{\eps}(x_0) \setminus B_{(1-\theta){\eps}}(x_0)$, 
and 
\begin{equation*} 
\begin{split}
\mathcal{F}(w_{\eps}, B_{\eps}(x_0)) \le &  (1+\eta)\big(\mathcal{F}(z_\eps, A_{\eps,x_0})  + \mathcal{F}(\bar{u}^{\rm surf}_{x_0},  \B_{\eps,x_0})\big) +   \frac{ M}{\eps^p} \Vert z_\eps - \bar{u}^{\rm surf}_{x_0} \Vert^p_{L^p( (A \setminus A')_{\eps,x_0})} 
 + \mathcal{L}^d(B_{\eps}(x_0))\eta.
\end{split}
\end{equation*}
We observe that $z_\eps = u_{\eps}$ outside $B_{(1-3\theta){\eps}}(x_0)$, by \eqref{1404201958} and the choice of  $s_\eps$.  Then, as done in Step~1, we may employ \eqref{eqs:blowupJumpPoints}(ii).  This gives us
a sequence $(\rho_\eps)_\eps\subset (0,+\infty)$ with  $\rho_\eps \to 0$ as $\eps\to 0$  such that 
\begin{align}\label{eq: rep2-new'}
\mathcal{F}&(w_\eps, B_{\eps}(x_0)) \le   (1+\eta)\big(\mathcal{F}(z_\eps,A_{\eps,x_0})  + \mathcal{F}(\bar{u}^{\rm surf}_{x_0},\B_{\eps,x_0})\big) +  \eps^{d-1}\rho_\eps +\gamma_d\,\eps^d\,\eta.
\end{align} 
We estimate the first terms in \eqref{eq: rep2-new'}.  We get by {\rm (${\rm H_1}$)},  {\rm (${\rm H_4}$)},  \eqref{1404201956}--\eqref{1404201958}, and the choice of $s_\eps$   that
\begin{equation}\label{eq: rep1-new'}
\mathcal{F}(z_\eps,A_{\eps,x_0})
\leq \mathbf{m}_{\mathcal{F}}(u, B_{s_\eps}(x_0)) + \eps^{d} + \beta \, \hd\big((\{u\neq u_{\eps}\} \cup J_u \cup J_{u_{\eps}}) \cap \partial B_{s_\eps}(x_0)\big) + \mathcal{F}(u_{\eps}, \B_{\eps,x_0}).
\end{equation}
By \eqref{eq: rep1XXX'}, \eqref{1404201943},  and the fact that $s_\eps \le (1-3\theta)\eps$  we thus deduce that
\begin{align}\label{1404202056}
\limsup_{\eps \to 0} \frac{\mathcal{F}(z_\eps, A_{\eps,x_0})}{\eps^{d-1}}&\leq \limsup_{\eps \to 0} \frac{\mathbf{m}_{\mathcal{F}}(u, B_{s_\eps}(x_0))}{\eps^{d-1}} + \beta\, \gamma_{d-1} \big( 1-(1-4\theta)^{d-1} \big)\notag
\\&
\leq (1-3 \theta)^{d-1} \limsup_{\eps \to 0} \frac{\mathbf{m}_{\mathcal{F}}(u, B_{ \eps}(x_0))}{\eps^{d-1}} + \beta\, \gamma_{d-1} \big( 1-(1-4\theta)^{d-1} \big),
\end{align}
and,  similarly to   \eqref{eq: rep1'}, 
\begin{equation}\label{1404202104}
 \limsup_{\eps \to 0}  \frac{\mathcal{F}(\bar{u}^{\rm surf}_{x_0}, \B_{\eps,x_0})}{\eps^{d-1}} \leq \beta\, \gamma_{d-1} \big( 1-(1-4\theta)^{d-1} \big).
\end{equation}
Collecting  \eqref{eq: rep2-new'}, \eqref{1404202056},   \eqref{1404202104}  and using $\rho_\eps \to 0$   we derive
\begin{align*}
 \limsup_{\eps \to 0}\frac{\mathcal{F}(w_\eps, B_\eps(x_0))}{\eps^{d-1}} & 
  \le (1+\eta) \Big( (1-3 \theta)^{d-1} \limsup_{\eps \to 0} \frac{\mathbf{m}_{\mathcal{F}}(u, B_{ \eps}(x_0))}{\eps^{d-1}}  + 2\beta  \gamma_{d-1} \big( 1-(1-4\theta)^{d-1} \big)\Big).
\end{align*}
Finally,  recalling that $w_{\eps}=\bar{u}^{\rm surf}_{x_0}$ in a neighborhood of $\partial B_{\eps}(x_0)$, and using the arbitrariness of $\eta$, $\theta>0$ we  obtain 
\begin{equation*}
\limsup_{\eps \to 0}\frac{\mathbf{m}_{\mathcal{F}}(\bar{u}^{\rm surf}_{x_0},B_\eps(x_0))}{\gamma_{d-1}\eps^{d-1} } \leq    \limsup_{\eps \to 0}\frac{\mathcal{F}(w_\eps, B_\eps(x_0))}{\gamma_{d-1}\eps^{d-1}}  \le \limsup_{\eps \to 0} \frac{\mathbf{m}_{\mathcal{F}}(u, B_{ \eps}(x_0))}{\gamma_{d-1}\eps^{d-1}}= \lim_{\varepsilon\to 0} \frac{\mathbf{m}_{\mathcal{F}}(u, B_{ \varepsilon}(x_0))}{\gamma_{d-1}\varepsilon^{d-1}}.
\end{equation*}
This  shows  ``$\ge$''   in \eqref{eq: PR-proof1}  and   concludes the proof. 
\end{proof}



\section{The $SBD^p$ case}\label{app: sbd}

 This  section  is devoted to the  analysis of the integral representation result  for  $\mathcal{F}\colon SBD^p(\Omega) \times \mathcal{B}(\Omega) \to  [0,+\infty)$ satisfying (H$_1$)--(H$_3$) and (H$_4^\prime$). This case has been addressed, for $d=2$, in \cite{Conti-Focardi-Iurlano:15}. On the one hand, the arguments there could be now generalized to general dimension  by virtue of  Theorem~\ref{th: kornSBDsmall}.  On the other hand,  as we are going to show, the result  can also be obtained with minor changes of our more general  strategy. 
 
%
%
We start by pointing out that, under (H$_4^\prime$), only competitors in $SBD^p$ may have finite energy.
In fact,  in view of  Proposition~\ref{prop: onlysbd}, in the present setting definition \eqref{eq: general minimization}  reads as  
\begin{equation}\label{eq: general minimizationsbd}
\mathbf{m}_{\mathcal{F}}(u,A) = \inf_{v \in SBD^p(\Omega)} \  \lbrace \mathcal{F}(v,A): \ v = u \ \text{ in a neighborhood of } \partial A \rbrace\,.
\end{equation}

Then, the following integral  representation  result holds.
\begin{theorem}[Integral representation in $SBD^p$]\label{theorem: SBD-representation}
Let $\Omega \subset \R^d$ be open, bounded with Lipschitz boundary and suppose that  $\mathcal{F}\colon SBD^p(\Omega)  \times \mathcal{B}(\Omega) \to  [0,+\infty)$ satisfies {\rm (${\rm H_1}$)}--{\rm (${\rm H_3}$)} and  {\rm (${\rm H^\prime_4}$)}.   Then  
$$\mathcal{F}(u,B) = \int_B f\big(x,u(x),\nabla u(x)\big)  \, {\rm d}x +    \int_{J_u\cap  B} g\big(x,u^+(x),u^-(x),\nu_u(x)\big)\, \mathrm{d} \mathcal{H}^{d-1}(x)$$
for all $u \in  SBD^p(\Omega)$  and   $B \in \mathcal{B}(\Omega)$, where $f$ is given  by
\begin{align*}
f(x_0,u_0,\xi) = \limsup_{\eps \to 0} \frac{\mathbf{m}_{\mathcal{F}}(\ell_{x_0,u_0,\xi},B_\eps(x_0))}{\gamma_d\eps^{d}}
\end{align*}
for all $x_0 \in \Omega$, $u_0 \in \R^d$, $\xi \in  \mathbb{M}^{d \times d} $,  and $\ell_{x_0,u_0,\xi}$ as in \eqref{eq: elastic competitor}, and $g$ is given by 
\begin{align*}
g(x_0,a,b,\nu) = \limsup_{\eps \to 0} \frac{\mathbf{m}_{\mathcal{F}}(u_{x_0,a,b,\nu},B_\eps(x_0))}{\gamma_{d-1}\eps^{d-1}}
\end{align*}
for all $  x_0  \in \Omega$,  $a,b \in \R^d$, $\nu \in \mathbb{S}^{d-1}$, and $u_{x_0,a,b,\nu}$ as in \eqref{eq: jump competitor}. 
\end{theorem} 

 The remainder  of this section  is devoted to the proof of Theorem \ref{theorem: SBD-representation} which follows along the lines of the proof of Theorem \ref{theorem: PR-representation} devised in Section \ref{sec: global method}.  First, the analogue of Lemma \ref{lemma: G=m} holds essentially with  the same proof.

\begin{lemma}\label{lemma: G=msbd}
Suppose that $\mathcal{F}$ satisfies {\rm (${\rm H_1}$)}--{\rm (${\rm H^\prime_3}$)}  and {\rm (${\rm H^\prime_4}$)}.  Let $u \in SBD^p(\Omega)$ and let   $\mu = \mathcal{L}^d\lfloor_{\Omega} + \mathcal{H}^{d-1}\lfloor_{J_u \cap \Omega}$. Then for $\mu$-a.e.\ $x_0 \in \Omega$ we have
 $$\lim_{\eps \to 0}\frac{\mathcal{F}(u,B_\eps(x_0))}{\mu(B_\eps(x_0))} =  \lim_{\eps \to 0}\frac{\mathbf{m}_{\mathcal{F}}(u,B_\eps(x_0))}{\mu(B_\eps(x_0))}.$$
\end{lemma}

\begin{proof} 
One can follow the same argument used to prove Lemma \ref{lemma: G=m} through \BBB Lemma  \ref{lemma: G=m-2}. \EEE 

 First,  we remark that \cite[Lemmas 4.2, 4.3]{Conti-Focardi-Iurlano:15} is proved under the assumptions {\rm (${\rm H^\prime_4}$)} and $u \in SBD^p(\Omega)$, hence it can be used \BBB to derive \eqref{eq:referee}. \EEE

Concerning Lemma~\ref{lemma: G=m-2},  as the lower bound of {\rm (${\rm H^\prime_4}$)} is stronger than the one of {\rm (${\rm H_4}$)},  the $GSBD$ compactness  result \cite[Theorem~1.1]{Crismale1} is still applicable. 
 First, this shows $v^\delta \in GSBD^p(\Omega)$.  Additionally, \eqref{eq: to show-flaviana2}, {\rm (${\rm H^\prime_4}$)}, and Proposition~\ref{prop: onlysbd} imply that the function $v^\delta$ belongs indeed to $SBD^p(\Omega)$. The rest of the proof remains unchanged, upon noticing that, under assumption  (${\rm H^\prime_4}$), \eqref{eq: to show-flaviana2-NNN} still holds. 
\end{proof}

 We now address the adaptions necessary for the bulk density. When $u \in SBD^p(\Omega)$,  we show that  the approximating sequence constructed in Lemma \ref{lemma: blow up}  satisfies  some additional properties. 

\begin{lemma}
Let $u \in SBD^p(\Omega)$. Let $\theta \in (0,1)$. For $\mathcal{L}^{d}$-a.e.\ $x_0 \in \Omega$  there exists a  family  $u_\eps \in SBD^p(B_\eps(x_0))$ such that \eqref{eq: blow up-new} holds, and additionally
\begin{align}\label{eq: blowup-new+}
{\rm (i)} &  \ \  \lim_{\eps \to 0}  \ \eps^{-(d+1)} \int_{B_{\eps}(x_0)}| u_\eps  - u|\, \mathrm{d}x = 0,\notag \\ 
{\rm (ii)} &  \ \  \lim_{\eps \to 0} \   \eps^{-d}\,\int_{ J_{u_\eps} }|[u_\eps]|\,\mathrm{d}\mathcal{H}^{d-1}= 0\,.
\end{align}
\end{lemma}

\begin{proof}
 Since   $u \in SBD(\Omega)$,  for $\mathcal{L}^d$-a.e.\ $x_0$  it holds that 
\begin{equation}\label{eq: SBDblowup}
\lim_{\eps \to 0} \eps^{-d} \int_{J_u\cap B_{\eps}(x_0)} |[u]| \, \mathrm{d} \mathcal{H}^{d-1}   = 0,
\end{equation}
and (see \cite[Theorem 7.4]{ACD}) that
\begin{equation}\label{eq:migliorblowup+}
\lim_{\eps \to 0} \eps^{-(d+1)} \int_{B_{\eps}(x_0)} \big|u- \bar{u}^{\rm bulk}_{x_0}\big| \, \mathrm{d}x = 0,
\end{equation}
where  for brevity we again let  $\bar{u}^{\rm bulk}_{x_0}  =   \ell_{x_0,u(x_0),\nabla u(x_0)}$. Hence, with Fubini's Theorem we can fix $s_\eps \in (1-\theta, 1-\frac\theta2)\eps$ so that $\hd(J_u \cap \partial B_{s_\eps}  (x_0)  )=0$ and
\begin{equation}\label{eq:buonincollamento-bulk}
\lim_{\eps \to 0}\eps^{-d} \int_{\partial B_{s_\eps}(x_0)}  \big|u- \bar{u}^{\rm bulk}_{x_0}\big| \, \mathrm{d}\hd = 0\,.
\end{equation}
We can now perform the same construction as in \eqref{eq: ueps-def} with $s_\eps$ in place of $(1-\theta)\eps$. Notice that  in  this case $u_\eps \in SBD^p(B_\eps(x_0))$. By arguing exactly as in the proof of Lemma~\ref{lemma: blow up}, we derive \eqref{eq: blow up-new}(i),(iii),(iv) while (ii) holds in  $B_{s_\eps}(x_0)$ and a fortiori in $B_{(1-\theta)\eps}(x_0)$. In particular,  this in combination  with  H\"older's inequality,   \eqref{eq:migliorblowup+},  and  $u=u_\eps$ in $B_\eps(x_0)\setminus B_{s_\eps}(x_0)$   yields \eqref{eq: blowup-new+}(i). 

To see  \eqref{eq: blowup-new+}(ii),  observe that, since $s_\eps/\eps$ is bounded from above and from below,  \eqref{eq: blow up-new}(ii),(iii),     the fact that $u_\eps\lfloor_{B_{s_\eps}(x_0)} \in W^{1,p}(B_{s_\eps}(x_0);\R^d)$ (see  \eqref{eq: ueps-def}), and  Korn's  inequality imply  that  
\[
\frac{u_\eps(x_0+s_\eps  \cdot  )- \bar{u}^{\rm bulk}_{x_0}(x_0+s_\eps  \cdot  )}{s_\eps}\to 0 \mbox { in }W^{1,p}\left(  B_{1}(0);  \R^d  \right)\,.
\]
Hence, by the trace inequality and  by scaling back to $B_{s_\eps}(x_0)$,  we obtain
\begin{equation*}
\lim_{\eps\to 0}s_\eps^{ -(d-1)}\int_{\partial B_{s_\eps}(x_0)}     \tfrac{1}{s_\eps}  \left|u_\eps-\bar{u}^{\rm bulk}_{x_0}\right| \, \mathrm{d}\hd = 0 \,.
\end{equation*}
With this, \eqref{eq:buonincollamento-bulk},  and the fact that $s_\eps/\eps$ is bounded from below   we then have
\[
\lim_{\eps \to 0} \   \eps^{-d}\int_{\partial B_{s_\eps}(x_0)} \big|u_\eps- u\big| \, \mathrm{d}\hd = 0\,.
\]
Hence, we get by construction of $u_\eps$ and \eqref{eq: SBDblowup}  that 
\[
\lim_{\eps \to 0} \eps^{-d} \int_{ J_{u_\eps}  } |[u_\eps]| \, \mathrm{d}\hd=\lim_{\eps \to 0} \eps^{-d} \int_{J_u\cap (B_{\eps}(x_0)\setminus B_{s_\eps}(x_0))} |[u]| \, \mathrm{d}\hd= 0,
\]
which concludes the proof. 
\end{proof}

With the above  lemma  at our disposal, we can deduce the asymptotic equivalence of the minimization problems \eqref{eq: general minimizationsbd} for $u$ and
$\bar{u}^{\rm bulk}_{x_0}  =    \ell_{x_0,u(x_0),\nabla u(x_0)}$.

\begin{lemma}\label{lemma: minsamesbd}
Suppose that $\mathcal{F}$ satisfies  {\rm (${\rm H_1}$)},  {\rm (${\rm H_3}$)},  and  {\rm (${\rm H^\prime_4}$)}.  Let $u \in SBD^p(\Omega)$.   Then for $\mathcal{L}^{d}$-a.e.\ $x_0 \in \Omega$  we have
\begin{align}\label{eq: PR-proof1-bulk_NNN}
  \lim_{\eps \to 0}\frac{\mathbf{m}_{\mathcal{F}}(u,B_\eps(x_0))}{\gamma_{d}\eps^{d}} =  \limsup_{\eps \to 0}\frac{\mathbf{m}_{\mathcal{F}}(\bar{u}^{\rm bulk}_{x_0},B_\eps(x_0))}{\gamma_{d}\eps^{d}}. 
\end{align}
\end{lemma}

\begin{proof}
We argue as in the proof of Lemma \ref{lemma: minsame}.

For the ''$\le$'' inequality,  take $u_\eps$  satisfying \eqref{eq: blow up-new} and \eqref{eq: blowup-new+}, and perform the same construction as in Lemma \ref{lemma: minsame}.  (Observe that the fundamental estimate also holds in this case, see Lemma \ref{lemma: fundamental estimate}.)   Notice that, in this case, we have by  {\rm (${\rm H^\prime_4}$)} and \eqref{eq: definition of sets}  that 
\begin{align*}
\mathcal{F}(u_\eps, \B_{\eps,x_0}) \le \beta \int_{\B_{\eps,x_0}} (1+  |e(u_\eps)|^p) \, \dx  + \beta \,\int_{J_{u_\eps} \cap \B_{\eps,x_0}}(1+|[u_\eps]|)\,\mathrm{d}\hd\,.
\end{align*}
 Thus,  using \eqref{eq: blow up-new}(iii),(iv) and \eqref{eq: blowup-new+}(ii), we still get \eqref{eq: rep1XXX}, and may deduce inequality ``$\le$''  in \eqref{eq: PR-proof1-bulk_NNN}.

For the ``$\ge$'' inequality, we start   by  observing  that, if $u_\eps$ also satisfies \eqref{eq: blowup-new+}, in addition to \eqref{1404201943'} we may require that
\begin{align}\label{1404201943'new}
\lim_{\eps \to 0}\varepsilon^{-d}\int_{\partial B_{s_\eps}(x_0)} \left|u^+-u^-_\eps\right| \,\mathrm{d}\hd= 0,
\end{align} 
 where $u^-_\eps$ and $u^+$ indicate the inner and outer traces at $\partial B_{s_\eps}(x_0)$, respectively.  We also have that \eqref{eq: rep1XXX}  holds,  as seen in the previous step.
We then perform the same construction as in Lemma  \ref{lemma: minsame}. In this case, inequality \eqref{eq: rep1-new'''} is replaced  by
\begin{equation*}
\begin{split}
\mathcal{F}(z_\eps,A_{\eps,x_0})
&\leq \mathbf{m}_{\mathcal{F}}(u, B_{s_\eps}(x_0)) + \eps^{d+1} + \beta \, \hd\big((\{u\neq u_{\eps}\} \cup J_u \cup J_{u_{\eps}}) \cap \partial B_{s_\eps}(x_0)\big)  \\
& \ \ \ +\beta \int_{\partial B_{s_\eps}(x_0)} \left|u^+-u^-_\eps\right|  \,\mathrm{d}\hd+ \mathcal{F}(u_{\eps}, \B_{\eps,x_0}),
\end{split}
\end{equation*}
so that, 
using \eqref{eq: rep1XXX}, \eqref{1404201943'},  \eqref{1404201943'new}, and the fact that $s_\eps \le (1-3\theta)\eps$  we are still in a position to deduce \eqref{eq: rep1-new}.  The rest of the argument remains unchanged and we obtain  inequality ``$\ge$'' in \eqref{eq: PR-proof1-bulk_NNN}.
\end{proof}

Similar changes have to be performed also  for the surface density.  We first deduce the analogue of Lemma~\ref{le:blowupJumpPoints}.  We again set  $\bar{u}^{\rm surf}_{x_0}  =  u_{x_0,u^+(x_0),u^-(x_0),\nu_u(x_0)}$ for brevity, see \eqref{eq: jump competitor}.   We also recall the notation in \eqref{eq: shift-not}. 
\begin{lemma}
Let $u \in SBD^p(\Omega)$ and $\theta \in (0,1)$. For $\hd$-a.e.\ $x_0 \in J_u$ there exists a family $u_\varepsilon  \in    SBD^p(B_\varepsilon(x_0))$ such that \eqref{eqs:blowupJumpPoints} holds, and additionally
\begin{align}\label{eqs:blowupJumpPoints+}
{\rm (i)} &  \ \  \lim_{\eps \to 0}  \ \eps^{-d} \int_{B_{\eps}(x_0)}|  u_\eps   - u|\, \mathrm{d}x = 0, \\ 
{\rm (ii)} &  \ \  \lim_{\eps \to 0} \   \eps^{-(d-1)}\,\int_{J_{u_\eps}\cap   E_{\eps,x_0} }|[u_\eps]| \, \mathrm{d}\mathcal{H}^{d-1}= |[ \bar{u}^{\rm surf}_{x_0}]|\, \hd(\Pi_0 \cap E) \quad \text{for all Borel sets } E \subset B_1(x_0),\notag
\end{align}
where $\Pi_0$ denotes the hyperplane passing  through  $x_0$ with normal $\nu_u(x_0)$.
\end{lemma}

\begin{proof}
 Since  $u \in SBD(\Omega) $,  for  $\hd$-a.e.\ $x_0 \in J_u$ it holds that 
\begin{equation}\label{eq:migliorblowup}
\lim_{\eps \to 0} \eps^{-d} \int_{B_{\eps}(x_0)} \big|u- \bar{u}^{\rm surf}_{x_0}\big| \, \mathrm{d}x = 0\,.
\end{equation}
Hence, by Fubini's Theorem we find  $s_\eps \in (1-\theta, 1-\frac\theta2)\eps$  such that \eqref{1304200828-new} holds, we have   $\hd(J_u \cap \partial B_{s_\eps} (x_0)  )=0$, and
\begin{equation}\label{eq:buonincollamento}
\lim_{\eps \to 0}\eps^{-(d-1)} \int_{\partial B_{s_\eps}(x_0)} \big|u- \bar{u}^{\rm surf}_{x_0}\big| \, \mathrm{d}\hd = 0\,.
\end{equation}
We can now perform the same construction as in \eqref{1104202008}  and  derive \eqref{eqs:blowupJumpPoints}.  In particular,  this  in combination  with  H\"older's inequality,  $p>1$,   \eqref{eq:migliorblowup},  and $u=u_\eps$ in $B_\eps(x_0)\setminus B_{s_\eps}(x_0)$  yields \eqref{eqs:blowupJumpPoints+}(i).

  To see  \eqref{eqs:blowupJumpPoints+}(ii),  observe that, since $s_\eps/\eps$ is bounded from above and from below,  \eqref{eqs:blowupJumpPoints}(ii),(iv),  $p>1$, the fact that   $u_\eps\lfloor_{B^{\pm}_{s_\eps}(x_0)} \in W^{1,p}(B^{\pm}_{s_\eps}(x_0);\R^d)$  (see \eqref{1104202008}), and Korn's inequality imply   that   
$$ 
\int_{ s_\eps^{-1} ( B^{\pm}_{s_\eps}(x_0) - x_0)} |u_\eps(x_0+s_\eps y )-\bar{u}^{\rm surf}_{x_0}|^p \, {\rm d}y + \int_{ s_\eps^{-1} ( B^{\pm}_{s_\eps}(x_0) - x_0)} |\nabla_y  u_\eps(x_0+s_\eps y )|^p \, {\rm d}y  \to 0\,.
$$ 
 Hence, by the trace inequality and  by scaling back to  $B^\pm_{s_\eps}(x_0)$,  we obtain
\begin{equation*}
\lim_{\eps\to 0}{s_\eps}^{-(d-1)}\int_{ \partial B^{\pm}_{s_\eps}(x_0)}\left|u_\eps- u^\pm(x_0)  \right|\, \mathrm{d}\hd = 0\,.
\end{equation*}
Since $s_\eps/\eps$ is bounded from below, we then get that
\begin{equation}\label{eq:traceestimate}
\lim_{\eps\to 0}\eps^{-(d-1)}\int_{ \partial B^{\pm}_{s_\eps}(x_0)}\left|u_\eps- u^\pm(x_0) \right|\, \mathrm{d}\hd = 0\,.
\end{equation}
  Given a  Borel  set   $E \subset B_1(x_0)$, we 
   define  $E^\eps= E\cap B_{\frac{s_\eps}\eps}(x_0)$ for every $\eps>0$.   Then by  \eqref{1104202008} and \eqref{eq:traceestimate}   we get 
\begin{equation}\label{eq:part1}
\lim_{\eps \to 0} \Big(   \eps^{-(d-1)}\, \int_{J_{u_\eps}\cap    E^\eps_{\eps,x_0}  } |[u_\eps]|\,\mathrm{d}\mathcal{H}^{d-1} -  |[ \bar{u}^{\rm surf}_{x_0}]|\,\hd(\Pi_0 \cap E^\eps) \Big) = 0\,. 
\end{equation}
By \eqref{eq:buonincollamento} and \eqref{eq:traceestimate}   we also have
\[
\lim_{\eps \to 0} \   \eps^{-(d-1)}\int_{\partial B_{s_\eps}(x_0)} \big|u^-_\eps- u^+\big| \, \mathrm{d}\hd = 0,
\]
 where $u^-_\eps$ and $u^+$ indicate the inner and outer traces at $\partial B_{s_\eps}(x_0)$, respectively.  Hence, by construction of $u_\eps$ in \eqref{1104202008} and   since  $u\in SBD(\Omega)$, we obtain
\[
\begin{split}
&\lim_{\eps \to 0} \Big(   \eps^{-(d-1)}\,\int_{J_{u_\eps}\cap  (E\setminus  E^\eps  )_{\eps,x_0} }|[u_\eps]|\,\mathrm{d}\mathcal{H}^{d-1}  - |[ \bar{u}^{\rm surf}_{x_0}]|\,\hd(\Pi_0 \cap( E\setminus E^\eps)) \Big)  \\
&= \lim_{\eps \to 0} \Big(   \eps^{-(d-1)}\,\int_{J_{u}\cap  (E\setminus E^\eps)_{\eps,x_0} }|[u]|\,\mathrm{d}\mathcal{H}^{d-1}   - |[ \bar{u}^{\rm surf}_{x_0}]|\,\hd(\Pi_0 \cap( E\setminus E^\eps)) \Big)  = 0\,.
\end{split}
\]
Combining with \eqref{eq:part1}, this concludes the proof  of \eqref{eqs:blowupJumpPoints+}(ii). 
\end{proof}

 With this lemma at hand, we can address the equivalence of minimization problems for the surface scaling. 

\begin{lemma}\label{lemma: minsame2sbd}
Suppose that $\mathcal{F}$ satisfies {\rm (${\rm H_1}$)}, {\rm (${\rm H_3}$)},  and  {\rm (${\rm H^\prime_4}$)}. Let   $u \in SBD^p(\Omega)$. Then for $\mathcal{H}^{d-1}$-a.e.\ $x_0 \in J_u$  we have
\begin{align}\label{eq: PR-proof1_NNN}
  \lim_{\eps \to 0}\frac{\mathbf{m}_{\mathcal{F}}(u,B_\eps(x_0))}{\gamma_{d-1}\eps^{d-1}} =  \limsup_{\eps \to 0}\frac{\mathbf{m}_{\mathcal{F}}(\bar{u}^{\rm surf}_{x_0},B_\eps(x_0))}{\gamma_{d-1}\eps^{d-1}}. 
 \end{align}
\end{lemma}

\begin{proof}
We argue as in the proof of  Lemma~\ref{lemma: minsame2}. 

For the "$\le$" inequality, take $u_\eps$ satisfying  \eqref{eqs:blowupJumpPoints} and \eqref{eqs:blowupJumpPoints+}, and perform the same construction as in  Lemma~\ref{lemma: minsame2}.   The estimates  \eqref{eq:repr1+++'} and \eqref{eq: rep1'} continue to hold, provided one replaces $\beta$ with the larger constant $\beta(1+|[\bar{u}^{\rm surf}_{x_0}]|)$. Then,  with  {\rm (${\rm H^\prime_4}$)}, \eqref{eqs:blowupJumpPoints}(iii),(iv), and \eqref{eqs:blowupJumpPoints+}(ii) we get
\begin{align}\label{eq: rep1XXX'new}
\limsup_{\eps \to 0} \frac{\mathcal{F}(u_\eps,  \B_{\eps,x_0})}{  \eps^{d-1} }  \le \beta (1+|[\bar{u}^{\rm surf}_{x_0}]|)\gamma_{d-1}\big( 1-(1-4\theta)^{d-1}  \big),
\end{align}
which is  the analogue of \eqref{eq: rep1XXX'}. This is  enough to derive inequality ``$\le$''    in \eqref{eq: PR-proof1_NNN}.    

For the reverse one, take again $u_\eps$ satisfying  \eqref{eqs:blowupJumpPoints} and \eqref{eqs:blowupJumpPoints+}. Then,  by \eqref{eqs:blowupJumpPoints+}(i),  in addition to \eqref{1404201943} we may also require that
\begin{align}\label{1404201943new}
\lim_{\eps \to 0}\varepsilon^{-(d-1)}\int_{\partial B_{s_\eps}(x_0)} \left|u^+-u^-_\eps\right|  \,\mathrm{d}\hd= 0,
\end{align} 
 where $u^-_\eps$ and $u^+$ indicate the inner and outer traces at $\partial B_{s_\eps}(x_0)$, respectively. 
We perform the same construction as in  Lemma~\ref{lemma: minsame2}.  In this case, inequality \eqref{eq: rep1-new'} is replaced  by
\begin{equation*}
\begin{split}
\mathcal{F}(z_\eps,A_{\eps,x_0})
&\leq \mathbf{m}_{\mathcal{F}}(u, B_{s_\eps}(x_0)) + \eps^{d} + \beta \, \hd\big((\{u\neq u_{\eps}\} \cup J_u \cup J_{u_{\eps}}) \cap \partial B_{s_\eps}(x_0)\big)  \\
& \ \ \ +\beta \int_{\partial B_{s_\eps}(x_0)} \left|u^+-u^-_\eps\right| \,\mathrm{d}\hd+ \mathcal{F}(u_{\eps}, \B_{\eps,x_0}),
\end{split}
\end{equation*}
so that, using \eqref{1404201943}, \eqref{eq: rep1XXX'new},   \eqref{1404201943new}, and the fact that $s_\eps \le (1-3\theta)\eps$  we deduce that
\begin{align*}
\limsup_{\eps \to 0} \frac{\mathcal{F}(z_\eps, A_{\eps,x_0})}{\eps^{d-1}}
\leq (1-3 \theta)^{d-1} \limsup_{\eps \to 0} \frac{\mathbf{m}_{\mathcal{F}}(u, B_{ \eps}(x_0))}{\eps^{d-1}} + \beta (1+|[\bar{u}^{\rm surf}_{x_0}]|)\, \gamma_{d-1} \big( 1-(1-4\theta)^{d-1} \big)\,.
\end{align*}
Then one concludes exactly in the same way, upon replacing $\beta$ in \eqref{1404202104} with the larger constant  $\beta (1+|[\bar{u}^{\rm surf}_{x_0}]|)$.
\end{proof}

\begin{proof}[Proof of Theorem \ref{theorem: SBD-representation}]
The result follows from Lemmas \ref{lemma: G=msbd}, \ref{lemma: minsamesbd}, and \ref{lemma: minsame2sbd}, arguing exactly as in the proof of Theorem \ref{theorem: PR-representation}.
\end{proof}

\section{The $GSBV^p$ case}\label{sec: appendix}

In this section we briefly remark that the strategy devised in this paper allows also to establish an integral representation result in the  space $GSBV^p(\Omega;\R^m)$ for $m \in \N$. We consider functionals  $\mathcal{F}\colon GSBV^p(\Omega;\R^m) \times \mathcal{B}(\Omega) \to  [0,+\infty)$ with the following general assumptions: 
\begin{itemize}
\item[($\mathbb{H}_1$)]  $\mathcal{F}(u,\cdot)$ is a Borel measure for any $u \in  GSBV^p(\Omega;\R^m)$, 
\item[($\mathbb{H}_2$)]  $\mathcal{F}(\cdot,A)$ is lower semicontinuous with respect to convergence in measure on $\Omega$ for any $A \in \mathcal{A}(\Omega)$,
\item[($\mathbb{H}_3$)]   $\mathcal{F}(\cdot, A)$ is local for any $A \in \mathcal{A}(\Omega)$, in the sense that if $u,v \in GSBV^p(\Omega;\R^m)$ satisfy $u=v$ a.e.\ in $A$, then $\mathcal{F}(u,A) = \mathcal{F}(v,A)$,
\item[($\mathbb{H}_4$)]  there exist $0 < \alpha  < \beta $ such that for any $u \in GSBV^p(\Omega;\R^m)$ and $B \in \mathcal{B}(\Omega)$  we have
$$\alpha \bigg(\int_{ B  } |\nabla u|^p  \dx    +   \mathcal{H}^{d-1}(J_u \cap B)\bigg) \le \mathcal{F}(u,B) \le \beta \bigg(\int_{ B } (1 + |\nabla u|^p)    \dx  +   \mathcal{H}^{d-1}(J_u \cap B)\bigg).$$ 
\end{itemize}

In this setting,    we replace  definition \eqref{eq: general minimization} by
\begin{equation}\label{eq: general minimizationsgbv}
\mathbf{m}_{\mathcal{F}}(u,A) = \inf_{v \in GSBV^p(\Omega;\R^m)} \  \lbrace \mathcal{F}(v,A): \ v = u \ \text{ in a neighborhood of } \partial A \rbrace\,.
\end{equation}
Moreover, as in  \eqref{eq: elastic competitor}--\eqref{eq: jump competitor}, we define the functions  $\ell_{x_0,u_0,\xi}(x) =  u_0 + \xi (x-x_0)$ and $u_{x_0,a,b,\nu}(x) = a$ on $\lbrace  (x-x_0) \cdot \nu > 0\rbrace$ and $u_{x_0,a,b,\nu}(x) = b$ on $\lbrace  (x-x_0) \cdot \nu < 0\rbrace$ for  $x_0 \in \Omega$, $u_0 \in \R^m$, $\xi \in \mathbb{M}^{m \times d}$,   $a,b \in \R^m$, and  $\nu \in \mathbb{S}^{d-1}$.

\begin{theorem}[Integral representation in $GSBV^p$]\label{theorem: PR-representation-gsbv}
Let $\Omega \subset \R^d$ be open, bounded with Lipschitz boundary, let $m \in \N$,  and suppose that  $\mathcal{F}\colon GSBV^p(\Omega;\R^m)  \times \mathcal{B}(\Omega) \to [0,+\infty)$ satisfies {\rm ($\mathbb{H}_1$)}--{\rm ($\mathbb{H}_4$)}. Then 
$$\mathcal{F}(u,B) = \int_B f\big(x,u(x),\nabla u(x)\big)  \, {\rm d}x +    \int_{J_u\cap  B} g\big(x,u^+(x),u^-(x),\nu_u(x)\big)\,  {\rm d}  \mathcal{H}^{d-1}(x)$$
for all $u \in  GSBV^p(\Omega;\R^m)$  and   $B \in \mathcal{B}(\Omega)$, where $f$ is given  by
\begin{align}\label{eq:fdef-gsbv}
f(x_0,u_0,\xi) = \limsup_{\eps \to 0} \frac{\mathbf{m}_{\mathcal{F}}(\ell_{x_0,u_0,\xi},B_\eps(x_0))}{\gamma_d\eps^{d}}
\end{align}
for all $x_0 \in \Omega$, $u_0 \in \R^m$, $\xi \in \mathbb{M}^{m \times d}$,  and $g$ is given by 
\begin{align}\label{eq:gdef-gsbv}
g(x_0,a,b,\nu) = \limsup_{\eps \to 0} \frac{\mathbf{m}_{\mathcal{F}}(u_{x_0,a,b,\nu},B_\eps(x_0))}{\gamma_{d-1}\eps^{d-1}}
\end{align}
for all $ x_0  \in \Omega$,  $a,b \in \R^m$, and $\nu \in \mathbb{S}^{d-1}$.
\end{theorem} 


We point out that integral representation results in $GSBV^p$ have been used in several contributions, see e.g.\    \cite{BacBraZep18, BacCicRuf19, barfoc, BarLazZep16, Caterina, focgelpon07}.   
They all rely on \cite{BFLM} along with a perturbation and truncation argument as follows: first, one considers the regularization $\mathcal{F}_\sigma(u) := \mathcal{F}(u) + \sigma\int_{J_u} |[u]|\, {\rm d}\mathcal{H}^{d-1}$, restricted to $u \in SBV^p(\Omega;\R^m)$. Then, the assumptions of the integral representation result in $SBV^p$  \cite{BFLM} are satisfied and one obtains a representation of $\mathcal{F}_\sigma$. In a second step, this representation is extended to $GSBV^p$ by a truncation argument which allows to approximate $GSBV^p$ functions by $SBV^p$ functions. Eventually,  by sending $\sigma \to 0$, an integral representation result for the original functional can be obtained. We refer to \cite[Theorem 4.3, Theorem 5.1]{Caterina} for details on this procedure. (We notice that in \cite{Caterina} a more general growth condition from above is allowed in the surface energy density, cf.\ \cite[assumption (1.4)]{Caterina}, analogous to the  one in (H$_4^\prime$)). With our result at hand, this method can be considerably simplified since no perturbation and truncation arguments are needed.

For the proof we need the following  Poincar\'e-type inequality, which can be directly deduced from Theorem \ref{th: kornSBDsmall}.

\begin{theorem}[Poincar\'e inequality for functions with small jump set]\label{th: poincareGSBV}
Let $\Omega \subset \R^d$ be a bounded Lipschitz domain and let $1 < p < +\infty$. Then there exists a constant $c = c(\Omega,p,m)>0$ such that for all  $u \in GSBV^p(\Omega; \R^m)$ there is a set of finite perimeter $\omega \subset \Omega$ with 
\begin{align*}
\mathcal{H}^{d-1}(\partial^* \omega) \le c\mathcal{H}^{d-1}(J_u), \ \ \ \ \mathcal{L}^d(\omega) \le c(\mathcal{H}^{d-1}(J_u))^{d/(d-1)}
\end{align*}
and $v \in W^{1,p}(\Omega;\R^m)$ such that $v= u$ on $\Omega \setminus \omega$ and
\begin{align}\label{eq: main estmain-Sobolevp}
\Vert \nabla v \Vert_{L^p(\Omega)} \le c \Vert \nabla u \Vert_{L^p(\Omega)}.
\end{align}
In particular, for all  $u \in GSBV^p(\Omega; \R^m)$ there is a constant $b\in \R^m$ such that
\begin{align*}
\Vert u - b\Vert_{L^{p}(\Omega \setminus \omega)} \le c \Vert \nabla u \Vert_{L^p(\Omega)}.
\end{align*}
\end{theorem}

\begin{proof} 
It suffices to consider the case $m=1$ and to prove \eqref{eq: main estmain-Sobolevp}. This can be obtained for instance by applying Theorem  \ref{th: kornSBDsmall} to the function $\bar u\colon\Omega \to \R^d$ defined as $\bar u:=(u,0, \dots, 0)$ and using the Sobolev-Korn inequality to get $\nabla v$ on the left-hand side.
\end{proof}

\begin{proof}[Proof of Theorem \ref{theorem: PR-representation-gsbv}]
We follow the proof of Theorem \ref{theorem: PR-representation} and only indicate briefly the necessary adaptions. First, we observe that a version of the fundamental estimate in Lemma \ref {lemma: fundamental estimate} holds true in $GSBV^p(\Omega;\R^m)$ by repeating the proof with ($\mathbb{H}_4$) in place of  ({H}$_4$). (We refer also to \cite[Proposition~3.1]{Braides-Defranceschi-Vitali}.)  Recall that the result follows by combining Lemmas \ref{lemma: G=m}, \ref{lemma: minsame}, and \ref{lemma: minsame2}.

\emph{Lemma \ref{lemma: G=m}:} The result is proved via \BBB \eqref{eq:referee} and Lemma \EEE \ref{lemma: G=m-2}. \BBB Notice indeed that the derivation of  \eqref{eq:referee}, as well as the argument ensuing therefrom \EEE are the same, up to using the growth condition ($\mathbb{H}_4$) instead of ({H}$_4$). (We  also refer  to \cite[Lemma 6]{BFLM} for the corresponding argument in $SBV^p$.)  In the proof of Lemma \ref{lemma: G=m-2}, due to the (stronger) lower bound in ($\mathbb{H}_4$) and Ambrosio's compactness theorem in $GSBV^p$ (see  \cite[Theorem~4.36]{Ambrosio-Fusco-Pallara:2000}),  one can ensure that the function $v^\delta$ defined in \eqref{eq: def of vdelta} now belongs to $GSBV^p(\Omega;\R^m)$. Then, the result follows with the same argument,  up to using Theorem \ref{th: poincareGSBV} in place of Theorem \ref{th: kornSBDsmall}.

\emph{Lemma \ref{lemma: minsame}:} With the fundamental estimate in $GSBV^p$ at hand, we can follow the proof of Lemma~\ref{lemma: minsame} for each $u \in GSBV^p(\Omega;\R^m)$  with  ($\mathbb{H}_4$) instead of ({H}$_4$). The  family  $(u_\eps)_\eps$  is defined as in Lemma \ref{lemma: blow up} (see \eqref{eq: ueps-def}) using Theorem \ref{th: poincareGSBV} in place of Theorem \ref{th: kornSBDsmall}.  
First, $u_\eps \in GSBV^p(B_\eps(x_0);\R^m)$ since $u \in GSBV^p(\Omega;\R^m)$, and $u_\eps\lfloor_{B_{(1-\theta)\eps}(x_0)} \in W^{1,p}(B_{(1-\theta)\eps}(x_0);\R^m)$. Observing that  
\begin{equation}\label{eq: blowupGSBV}
\lim_{\eps \to 0} \ \eps^{-d} \int_{B_{\eps}(x_0)} \big| \nabla u(x) - \nabla u(x_0)\big|^p \, \mathrm{d}x = 0
\end{equation}
for $\mathcal{L}^d$-a.e.\ $x_0 \in \Omega$ (as $\nabla u \in L^p(\Omega;  \mathbb{M}^{m\times d}  )$) and using \eqref{eq: main estmain-Sobolevp}, we further get
$$ \lim_{\eps \to 0} \ \eps^{-d} \int_{B_{(1-\theta)\eps}(x_0)} \big| \nabla u_\eps(x) - \nabla u(x_0)\big|^p \, \mathrm{d}x = 0. $$ 
With \eqref{eq: ueps-def}, and using again \eqref{eq: blowupGSBV}, we deduce that \eqref{eq: blow up-new}(iii) can be improved to   
\begin{align*}
\lim_{\eps \to 0} \ \eps^{-d} \int_{B_{\eps}(x_0)} \big| \nabla u_\eps(x) - \nabla u(x_0)\big|^p \, \mathrm{d}x = 0.
\end{align*}
 This adaption is enough to redo the proof of Lemma \ref{lemma: minsame} in the present situation.

\emph{Lemma \ref{lemma: minsame2}:} Here, we can follow the proof of Lemma \ref{lemma: minsame2} for each $u \in GSBV^p(\Omega;\R^m)$  with  ($\mathbb{H}_4$) instead of ({H}$_4$), and Theorem~\ref{th: poincareGSBV} in place of Theorem \ref{th: kornSBDsmall}.  The  family  $(u_\eps)_\eps$  defined in Lemma~\ref{le:blowupJumpPoints}  needs to satisfy $u_\eps \in GSBV^p(B_\eps(x_0);\R^m)$ and \eqref{eqs:blowupJumpPoints}(iv) needs to be improved to   
\begin{align}\label{eq: blow up-new-gsbv2} 
\lim_{\eps \to 0} \ \eps^{-(d-1)} \int_{B_{\eps}(x_0)} \big| \nabla u_\eps\big|^p \, \mathrm{d}x = 0.
\end{align}
First, we use \eqref{1104202008} to see that $u_\eps \in GSBV^p(B_\eps(x_0);\R^m)$. Arguing as in the proof of \eqref{eqs:blowupJumpPoints}(iv),  with Theorem \ref{th: poincareGSBV} at hand, we obtain
\begin{equation*}
\lim_{\eps\to 0}{\eps}^{-(d-1)}\int_{B^{\pm}_{s_\eps}(x_0)}| \nabla u_\eps|^p\,  \mathrm{d}x  = 0.
\end{equation*}
This along with $\lim_{\eps \to 0} \ \eps^{-(d-1)} \int_{B_{\eps}(x_0)} | \nabla u|^p \, \mathrm{d}x = 0$ for $\mathcal{H}^{d-1}$-a.e.\ $x_0 \in J_u$ and \eqref{1104202008} concludes the proof of \eqref{eq: blow up-new-gsbv2}.
\end{proof}

\section*{Acknowledgements} 
\noindent For this project VC has received funding from the European Union’s Horizon 2020 research and innovation programme under the Marie Skłodowska-Curie grant agreement No.\ 793018.
MF is supported by the DFG project FR 4083/1-1 and  by the Deutsche Forschungsgemeinschaft (DFG, German Research Foundation) under Germany's Excellence Strategy EXC 2044 -390685587, Mathematics M\"unster: Dynamics--Geometry--Structure. 
The work of FS is part of the project “Variational methods for stationary and evolution problems with singularities and interfaces” PRIN 2017 financed by the Italian Ministry of Education, University, and Research.

\bibliographystyle{siam}
\bibliography{biblioIntRep}

\end{document}